\documentclass[11pt]{amsart}
\usepackage{latexsym,amssymb,amsfonts,amsmath,graphicx,url,color, verbatim}
\usepackage[all]{xy}
\usepackage{verbatim}
\usepackage{accents, wasysym}
\usepackage[usenames, dvipsnames]{xcolor}
\usepackage{ytableau}

\newcommand{\blank}{\phantom{2}}
\usepackage{tikz, calc, ifthen}
\usetikzlibrary{calc, shapes, backgrounds,arrows,positioning}
\tikzset{>=stealth',
  head/.style = {fill = white, text=black}, 
  pil/.style={->,thick},
  junct/.style = {draw,circle,inner sep=0.5pt,outer sep=0pt, fill=black}
  }
\definecolor{light-gray}{gray}{0.85}
\definecolor{dark-gray}{gray}{0.25}

\usepackage{tikz,tkz-graph}
\newcounter{x}
\newcounter{y}
\newcounter{z}

\newcommand\xaxis{210}
\newcommand\yaxis{-30}
\newcommand\zaxis{90}

\newcommand\topside[3]{
  \fill[fill=white, draw=black,shift={(\xaxis:#1)},shift={(\yaxis:#2)},
  shift={(\zaxis:#3)}] (0,0) -- (30:1) -- (0,1) --(150:1)--(0,0);
}

\newcommand\leftside[3]{
  \fill[fill=light-gray, draw=black,shift={(\xaxis:#1)},shift={(\yaxis:#2)},
  shift={(\zaxis:#3)}] (0,0) -- (0,-1) -- (210:1) --(150:1)--(0,0);
}

\newcommand\rightside[3]{
  \fill[fill=dark-gray, draw=black,shift={(\xaxis:#1)},shift={(\yaxis:#2)},
  shift={(\zaxis:#3)}] (0,0) -- (30:1) -- (-30:1) --(0,-1)--(0,0);
}

\newcommand\cube[3]{
  \topside{#1}{#2}{#3} \leftside{#1}{#2}{#3} \rightside{#1}{#2}{#3}
}

\newcommand\planepartition[1]{
 \setcounter{x}{-1}
  \foreach \a in {#1} {
    \addtocounter{x}{1}
    \setcounter{y}{-1}
    \foreach \b in \a {
      \addtocounter{y}{1}
      \setcounter{z}{-1}
      \foreach \c in {1,...,\b} {
        \addtocounter{z}{1}
        \cube{\value{x}}{\value{y}}{\value{z}}
      }
    }
  }
}
\newcommand\topsidelabel[3]{
  \fill[fill=white, draw=black,shift={(\xaxis:#1)},shift={(\yaxis:#2)},
  shift={(\zaxis:#3)}] (0,0) -- (30:1) -- (0,1) --(150:1)--(0,0);
    \pgfmathtruncatemacro{\label}{#3+1};
  \node[red,shift={(\xaxis:#1)},shift={(\yaxis:#2)},
  shift={(\zaxis:(#3+.5))},xshift=.0,yshift=.0,rotate=-60] {\label};
}

\newcommand\leftsidelabel[3]{
  \fill[fill=light-gray, draw=black,shift={(\xaxis:#1)},shift={(\yaxis:#2)},
  shift={(\zaxis:#3)}] (0,0) -- (0,-1) -- (210:1) --(150:1)--(0,0);
  \pgfmathtruncatemacro{\label}{#1+1};
    \node[blue,shift={(\xaxis:(#1+.5))},shift={(\yaxis:#2)},
  shift={(\zaxis:#3)},rotate=60] {\label};
}

\newcommand\rightsidelabel[3]{
  \fill[fill=dark-gray, draw=black,shift={(\xaxis:#1)},shift={(\yaxis:#2)},
  shift={(\zaxis:#3)}] (0,0) -- (30:1) -- (-30:1) --(0,-1)--(0,0);
  \pgfmathtruncatemacro{\label}{#2+1};
  \node[green,shift={(\xaxis:#1)},shift={(\yaxis:(#2+.5))},
  shift={(\zaxis:#3)},rotate=-60] {\label};
}

\newcommand\cubelabel[3]{
  \topsidelabel{#1}{#2}{#3} \leftsidelabel{#1}{#2}{#3} \rightsidelabel{#1}{#2}{#3}
}

\newcommand\planepartitionlabel[1]{
 \setcounter{x}{-1}
  \foreach \a in {#1} {
    \addtocounter{x}{1}
    \setcounter{y}{-1}
    \foreach \b in \a {
      \addtocounter{y}{1}
      \setcounter{z}{-1}
      \foreach \c in {1,...,\b} {
        \addtocounter{z}{1}
        \cubelabel{\value{x}}{\value{y}}{\value{z}}
      }
    }
  }
}

\newcommand\xx{-1.5}
\newcommand\xxx{1}
\newcommand\yy{1}
\newcommand\yyy{1}
\newcommand\zz{0}
\newcommand\zzz{.75}

\newcommand\topsideQ[3]{
\draw ($ #1*(\xx,\xxx)+ #2*(\yy,\yyy) + #3*(\zz,\zzz) $) node[fill=black,draw,circle,inner sep=1pt] {} -- ++(\xx,\xxx) node[fill=black,draw,circle,inner sep=1pt] {} -- ++(\yy,\yyy) node[fill=black,draw,circle,inner sep=1pt] {} -- ++(-\xx,-\xxx) node[fill=black,draw,circle,inner sep=1pt] {} -- ++(-\yy,-\yyy) node[fill=black,draw,circle,inner sep=1pt] {};
}

\newcommand\leftsideQ[3]{
 \draw ($ #1*(\xx,\xxx)+ #2*(\yy,\yyy) + #3*(\zz,\zzz) $) -- ++(\zz,\zzz) -- ++(\yy,\yyy) -- ++(-\zz,-\zzz) -- ++(-\yy,-\yyy);
}

\newcommand\rightsideQ[3]{
  \draw ($ #1*(\xx,\xxx)+ #2*(\yy,\yyy) + #3*(\zz,\zzz) $) -- ++(\zz,\zzz) -- ++(\xx,\xxx) -- ++(-\zz,-\zzz) -- ++(-\xx,-\xxx);
}
\newcommand\aaaQ[3]{
\draw ($ #1*(\xx,\xxx)+ #2*(\yy,\yyy) +  #3*(\zz,\zzz) + (\zz,\zzz) $) node[fill=black,draw,circle,inner sep=1pt] {} -- ++(\xx,\xxx) node[fill=black,draw,circle,inner sep=1pt] {} -- ++(\yy,\yyy) node[fill=black,draw,circle,inner sep=1pt] {} -- ++(-\xx,-\xxx) node[fill=black,draw,circle,inner sep=1pt] {} -- ++(-\yy,-\yyy) node[fill=black,draw,circle,inner sep=1pt] {};
}
\newcommand\bbbQ[3]{
 \draw ($ #1*(\xx,\xxx)+ #2*(\yy,\yyy) + #3*(\zz,\zzz) + (\xx,\xxx) $) -- ++(\zz,\zzz) -- ++(\yy,\yyy) -- ++(-\zz,-\zzz) -- ++(-\yy,-\yyy);
}
\newcommand\cccQ[3]{
  \draw ($ #1*(\xx,\xxx)+ #2*(\yy,\yyy) + #3*(\zz,\zzz) + (\yy,\yyy) $) -- ++(\zz,\zzz) -- ++(\xx,\xxx) -- ++(-\zz,-\zzz) -- ++(-\xx,-\xxx);
}

\newcommand\cubeQ[3]{
  \topsideQ{#1}{#2}{#3} \leftsideQ{#1}{#2}{#3} \rightsideQ{#1}{#2}{#3} \aaaQ{#1}{#2}{#3} \bbbQ{#1}{#2}{#3} \cccQ{#1}{#2}{#3}
}

\newcommand\planepartitionQ[1]{
 \setcounter{x}{-1}
  \foreach \a in {#1} {
    \addtocounter{x}{1}
    \setcounter{y}{-1}
    \foreach \b in \a {
      \addtocounter{y}{1}
      \setcounter{z}{-1}
      \foreach \c in {1,...,\b} {
        \addtocounter{z}{1}
        \cubeQ{\value{x}}{\value{y}}{\value{z}}
      }
    }
  }
}

\newtheorem{theorem}{Theorem}[section]
\newtheorem{proposition}[theorem]{Proposition}
\newtheorem{problem}[theorem]{Problem}
\newtheorem{lemma}[theorem]{Lemma}
\newtheorem*{squarethm}{Theorem~3.1}

\newtheorem*{height2thm}{Theorem~3.2}
\newtheorem{corollary}[theorem]{Corollary}
\newtheorem*{kpro_resonance_cor}{Theorem~2.2}
\newtheorem*{row_resonance_cor}{Theorem~3.10}
\newtheorem*{equiv_biject}{Theorem~4.5}
\newtheorem*{very_overworked}{Corollary~4.9}
\newtheorem*{somewhat_overworked}{Corollary~4.10}
\newtheorem*{3Dprorowthm}{Theorem~3.26}
\newtheorem*{our_cfdf}{Theorem~4.12}

\newtheorem{conjecture}[theorem]{Conjecture}
\newtheorem{fact}[theorem]{Fact}
\theoremstyle{definition}
\newtheorem{definition}[theorem]{Definition}

\theoremstyle{remark}
\newtheorem{remark}[theorem]{Remark}
\numberwithin{equation}{section}
\DeclareMathOperator*{\pro}{Pro}
\DeclareMathOperator*{\row}{\rm Row}
\DeclareMathOperator{\kpro}{{\it K}-Pro}
\DeclareMathOperator{\kbk}{{\it K}-BK}
\DeclareMathOperator*{\rk}{{\rm rk}}
\newcommand{\inc}[2]{\mathrm{Inc}^{#2}(#1)}

\newcommand*{\Scale}[2][4]{\scalebox{#1}{$#2$}}

\title[Resonance of plane partitions]{Resonance in orbits of plane partitions \\ and increasing tableaux}
\author{Kevin Dilks}
\address[KD]{Department of Mathematics \\ North Dakota State University \\ Fargo, ND 58102 \\ USA}
\email{kevin.dilks@ndsu.edu}

\author{Oliver Pechenik}
\address[OP]{Department of Mathematics \\ Rutgers University \\ Piscataway, NJ 08854 \\ USA}
\email{oliver.pechenik@rutgers.edu}

\author{Jessica Striker}
\address[JS]{Department of Mathematics \\ North Dakota State University \\ Fargo, ND 58102 \\ USA}
\email{jessica.striker@ndsu.edu}

\begin{document}

\begin{abstract}
We introduce a new concept of resonance on discrete dynamical systems. This concept formalizes the observation that, in various combinatorially-natural cyclic group actions, orbit cardinalities are all multiples of divisors of a fundamental frequency. %Our prototypical example of this phenomenon is B.~Wieland's gyration action on alternating sign matrices.
Our main result is an equivariant bijection between plane partitions in a box (or order ideals in the product of three chains) under rowmotion and increasing tableaux under $K$-promotion. Both of these actions were observed to have orbit sizes that were small multiples of divisors of an expected orbit size, and we show this is an instance of resonance, as $K$-promotion cyclically rotates the set of labels appearing in the increasing tableaux. 
We extract a number of corollaries from this equivariant bijection, including 
a strengthening of a theorem of [P.~Cameron--D.~Fon-der-Flaass '95] and several new results on the order of $K$-promotion. Along the way, we adapt the proof of the conjugacy of promotion and rowmotion from [J.~Striker--N.~Williams '12] to give a generalization in the setting of $n$-dimensional lattice projections. Finally we discuss known and conjectured examples of resonance relating to alternating sign matrices and fully-packed loop configurations.
\end{abstract}

\maketitle

\tableofcontents

\section{Introduction}
\label{sec:intro}
%The introduction to this paper is in two parts. The first subsection defines resonance and gives our prototypical example on alternating sign matrices. The second subsection describes our main results; these include two instances of resonance (on plane partitions and increasing tableaux), an equivariant bijection between these two sets with a number of new consequences, and a higher-dimensional analogue of N.~Williams and the third author's result on the equivariance of (poset-)promotion and rowmotion~\cite{prorow}.

%\subsection{Resonance}\label{sec:resonance}

We introduce the following concept of \emph{resonance}\footnote{The mathematically precise definition of resonance given here is new, though the phenomenon has been discussed by various people over the past several years, in particular, at the 2015 ``Dynamical Algebraic Combinatorics'' workshop at the American Institute of Mathematics where work on this paper began. Thanks to J.~Propp for coining the term ``resonance'' which so nicely encapsulates the idea.}.

\begin{definition}
Suppose $G = \langle g \rangle$ is a cyclic group acting on a set $X$, $\mathcal{C}_\omega = \langle c \rangle$ a cyclic group of order $\omega$ acting nontrivially on a set $Y$, and $f : X \to Y$ a surjection. We say the triple $(X, G, f)$ exhibits {\bf resonance with frequency $\omega$} if, for all $x \in X$, $c \cdot f(x) = f(g \cdot x)$, that is, the following diagram commutes:
\end{definition}

\begin{center}
\begin{tikzpicture}
\node (A) {$X$};
\node[right=2 of A] (B) {$X$};
\node[below =2 of A] (C) {$Y$};
\node[below right = 2 and 2 of A] (D) {$Y$};
\path (A) edge[pil] node[below]{$g \cdot$} (B);
\path (A) edge[pil] node[right]{$f$} (C);
\path (B) edge[pil] node[right]{$f$} (D);
\path (C) edge[pil] node[below]{$c \cdot$} (D);
\end{tikzpicture}
\end{center}

%\jessica{Make sure this diagram is in the right spot.}

In our examples, $Y$ will be either a set of combinatorial objects drawn in the plane with $c$ acting by rotation 
or a set of words with $c$ acting by a cyclic shift.  Resonance is a \emph{pseudo-periodicity} property of the $G$-action, in that the resonant frequency $\omega$ is generally less than the order of the $G$-action.
Note that $(X, G, {\rm id}_X)$ satisfies the definition of resonance with frequency $|G|$; we call this an instance of \emph{trivial resonance}.
In general, if a system exhibits resonance with frequency $\omega$, then it also exhibits resonance with frequency any multiple of $\omega$; hence one is most interested in finding resonances with small $\omega$.

We think of the property of resonance as somewhat analogous to the \emph{cyclic sieving phenomenon} (introduced by V.~Reiner--D.~Stanton--D.~White \cite{reiner.stanton.white}, generalizing the $q=-1$ phenomenon of J.~Stembridge \cite{stembridge}) and the \emph{homomesy property} (isolated by J.~Propp--T.~Roby \cite{propp.roby}, inspired by observations of D.~Panyushev \cite{panyushev}) in being a somewhat subtle ``niceness'' property of a cyclic group action.  We suspect that the phenomenon of resonance, like those of cyclic sieving and homomesy, is significantly more common than previously realized. Heuristically, one is led to suspect the presence of resonance in a system by observing that many orbit cardinalities are multiples or divisors (or multiples of divisors) of $\omega$.

%\jessica{Insert example here with increasing tableaux - need to define increasing tableaux as well}

This paper centers around two new examples of resonance on \emph{increasing tableaux} under \emph{$K$-promotion} and \emph{plane partitions} under \emph{rowmotion}, as well as a new equivariant bijection relating these phenomena. Here we summarize our main results, the first in greater detail to clarify the definition of resonance. See the referenced sections for relevant definitions.

%In Section~\ref{sec:increasing_tableaux}, we prove our first example of resonance, which involves $K$-\emph{promotion} on \emph{increasing tableaux}. We state this result and related definitions here briefly to clarify the definition of resonance given above; see Section~\ref{sec:increasing_tableaux} for full details.
%Our other main object of study is $K$-promotion on increasing tableaux.
%Identify a partition $\lambda$ with its Young diagram. (Throughout this paper, we use the English orientation on Young diagrams.) 
An {\bf increasing tableau} of partition shape $\lambda$ is a filling of $\lambda$ with positive integers such that labels strictly increase from left to right across rows and from top to bottom down columns. %An example appears in Figure~\ref{fig:an_increasing_tableau}. 
Denote as $\inc{\lambda}{q}$ the set of all increasing tableaux of shape $\lambda$ with entries at most $q$. %(In contrast to other definitions that have appeared in the literature, we do not assume here that every integer between $1$ and $q$ appears.) 
Define the {\bf binary content} of an increasing tableau $T \in \inc{\lambda}{q}$ to be the sequence ${\rm Con}(T) =(a_1, a_2, \dots, a_q)$, where $a_i = 1$ if $i$ is an entry of $T$ and $a_i = 0$ if it is not.

$K$\textbf{-promotion}, which we define in Section~\ref{sec:incdef} and denote as $\kpro$, was first described by the second author \cite{pechenik}. It is an variant of M.-P.~Sch\"utzenberger's promotion operator \cite{schutzenberger} built on the $K$-\emph{jeu de taquin} that was introduced by H.~Thomas--A.~Yong \cite{thomas.yong:V} to study $K$-theoretic Schubert calculus. $K$-promotion has been further studied in \cite{pressey.stokke.visentin, rhoades, bloom.pechenik.saracino}. 

%This was the sentence in the intro previously:
%Let $\inc{\lambda}{q}$ denote the set of increasing tableaux of shape $\lambda$ and entries at most $q$, and let {\rm Con} be the binary content map.

In Section 2.2, we prove the following, our first  result on resonance.
\begin{kpro_resonance_cor}
$(\inc{\lambda}{q}, \langle \kpro \rangle, {\rm Con})$ exhibits resonance with frequency~$q$.
\end{kpro_resonance_cor}

An example is shown in Figure~\ref{fig:resonant_tab}.

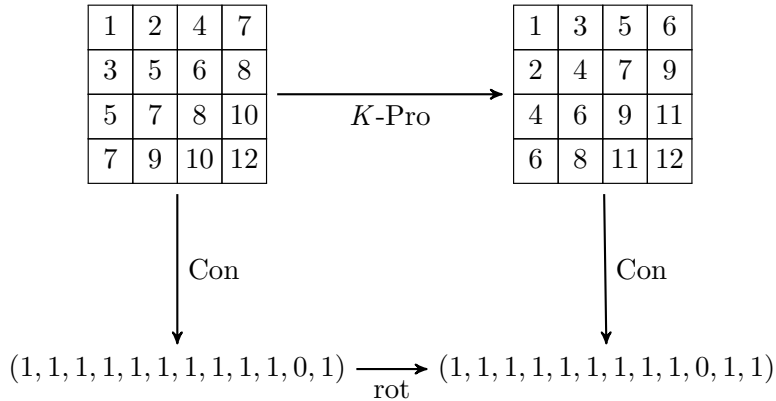
\begin{figure}[h]
\begin{center}
\begin{tikzpicture}
\node (A) {\ytableaushort{1247,3568,578{10},79{10}{12}}};
\node[right=3 of A] (B) {\ytableaushort{1356,2479,469{11},68{11}{12}}};
\node[below =2 of A] (C) {$(1,1,1,1,1,1,1,1,1,1,0,1)$};
\node[below right = 2 and 2 of A] (D) {$(1,1,1,1,1,1,1,1,1,0,1,1)$};
\path (A) edge[pil] node[below]{$\kpro$} (B);
\path (A) edge[pil] node[right]{${\rm Con}$} (C);
\path (B) edge[pil] node[right]{${\rm Con}$} (D);
\path (C) edge[pil] node[below]{${\rm rot}$} (D);
\end{tikzpicture}
\end{center}
\caption{An increasing tableau in $\inc{4 \times 4}{12}$ and its image under $\kpro$, along with the map to the binary content of each.}
\label{fig:resonant_tab}
\end{figure}

The $K$-promotion orbit of the depicted increasing tableaux in $\inc{4 \times 4}{12}$ has cardinality $36$. The binary content (written below each tableau) is, however, of order $12$ under cyclically shifting (denote this action as ${\rm rot}$). We will see in Lemma~\ref{prop:content_cycling} that the diagram of Figure~\ref{fig:resonant_tab} commutes. This illustrates the result that $(\inc{4\times 4}{12}, \langle \kpro \rangle, {\rm Con})$ exhibits resonance with frequency $12$, since while $\kpro^{12}(T)\neq T$ for either tableau in the figure, ${\rm rot}^{12}({\rm Con}(T))={\rm Con}(T)$ for all $T\in\inc{4 \times 4}{12}$.

%\subsection{Main results}
%\label{subsec:intro_main_results}

\textbf{Rowmotion}, which we define in Section~\ref{subsec:row}, has attracted much attention since it was first studied (under another name) by A.~Brouwer--A.~Schrijver \cite{brouwer.schrijver} in 1974; see for example \cite{fonderflaass:by_himself, fonderflaass, panyushev, prorow, armstrong.stump.thomas, rush.shi, propp.roby, rush.wang}. More recently, several authors have studied a \emph{birational} lift of rowmotion \cite{einstein.propp, grinberg.roby:1, grinberg.roby:2}, with some relations to Zamolodchikov periodicity.

Let $J(\mathbf{a}\times\mathbf{b}\times\mathbf{c})$ denote the set of \textbf{plane partitions} inside an $a\times b\times c$ box and $\row$ denote rowmotion; see Section~\ref{sec:resonance_pp} for the definitions of $X_{\max}$ and~D. Our second main resonance result is the following.
%a $\{0,1\}$-vector defined in  Section~\ref{sec:resonance_pp}. Our first main result on rowmotion is the following.
\begin{row_resonance_cor}
% Let $D$ be the conjugating toggle group element between rowmotion and promotion given in \cite[Theorem~5.4]{prorow}.
%Then
$(J(\mathbf{a}\times\mathbf{b}\times\mathbf{c}), \langle \row \rangle, X_{\max}\circ D)$ exhibits resonance with frequency~$a+b+c-1$.
\end{row_resonance_cor}

To better study plane partitions, we introduce and develop the machinery of \emph{affine hyperplane toggles} and \emph{$n$-dimensional lattice projections}, including a higher-dimensional analogue of N.~Williams and the third author's result on the equivariance of (poset-)promotion and rowmotion \cite{prorow}. We obtain a large family of toggling actions $\{ \pro_{\pi,v}^{\sigma} \}$ whose orbit structures are equivalent to that of rowmotion. See Sections~\ref{subsec:3Dprorow} and~\ref{subsec:provpi} for further details.

\begin{3Dprorowthm}
Let $P$ be a finite poset with an $n$-dimensional lattice projection $\pi$. Let  $v=(v_1,v_2,v_3,\ldots, v_{n})$ and $w=(w_1,w_2,w_3,\ldots, w_{n})$, where $v_j, w_j\in\{\pm 1\}$. Finally suppose that $\sigma:\mathrm{supp}(P,\pi,v)\rightarrow\mathrm{supp}(P,\pi,v)$ and $\tau:\mathrm{supp}(P,\pi,w)\rightarrow\mathrm{supp}(P,\pi,w)$ are bijections.
Then there is an equivariant bijection
between $J(P)$ under $\pro_{\pi,v}^{\sigma}$ and $J(P)$ under $\pro_{\pi,w}^{\tau}$.
\end{3Dprorowthm}

This similarity of Theorems 2.2 and 3.10 leads us to establish an equivariant bijection between plane partitions under rowmotion and increasing tableaux under $K$-promotion.
\begin{equiv_biject}
$J({\bf a} \times {\bf b} \times {\bf c})$ under $\row$ is in equivariant bijection with $\inc{a \times b}{a+b+c-1}$ under $\kpro$.
\end{equiv_biject}
Part of our approach to establishing this equivariant bijection involves the reinterpretation of $K$-promotion in terms of $K$-Bender-Knuth involutions, which we introduce; see Proposition~\ref{prop:kbk}. 
%\begin{kbk_prop}
%\label{prop:kbk}
%For $T \in \inc{\lambda}{q}$, $\kpro(T) = \kbk_{q-1} \circ \dots \circ \kbk_1 (T)$.
%\end{kbk_prop}
We also extend, in Section~\ref{sec:descent_cycling}, a result of B.~Rhoades on \emph{descent cycling} to the $K$-promotion setting.

We obtain a variety of corollaries of this equivariant bijection. Many of these corollaries are new proofs of previously discovered results on the order of $\row$ and $\kpro$. We highlight here only those results that are new.

\begin{very_overworked}
The order of $\kpro$ on $\inc{a \times b}{a + b}$ is $a+b$.
\end{very_overworked}

\begin{somewhat_overworked}
The order of $\kpro$ on $\inc{a \times b}{a + b+1}$ is $a+b+1$.
\end{somewhat_overworked}

We also obtain the following  strengthening of a theorem of P.~Cameron--D.~Fon-der-Flaass \cite[Theorem~6(a)]{fonderflaass}. The original theorem had the more stringent hypothesis $c > ab -a - b  + 1$.

\begin{our_cfdf}
If $a+b+c-1$ is prime and $c > \frac{2ab-2}{3} - a - b +2$, then the cardinality of every orbit of $\row$ on $J({\bf a} \times {\bf b} \times {\bf c})$ is a multiple of $a+b+c-1$.
\end{our_cfdf}

\textbf{The rest of this paper is structured as follows.}
In Section~\ref{sec:increasing_tableaux}, we recall the $K$-promotion operator on increasing tableaux and establish a number of new properties (including resonance) that we will use. In Section~\ref{sec:toggle}, we establish resonance of plane partitions under rowmotion and extend machinery developed by N.~Williams and the third author \cite{prorow} to introduce the family of toggle group actions $\{ \pro_{\pi, v}^\sigma \}$ and show that each $\pro_{\pi, v}^\sigma$ acts with the same cycle structure as rowmotion. In Section~\ref{sec:pp_inctab}, we give an equivariant bijection between increasing tableaux under $K$-promotion and plane partitions under $\pro_{(1,1,-1)}$ and $\row$. We then extract a number of corollaries from this equivariant bijection, including new proofs of theorems of A.~Brouwer--A.~Schrijver \cite{brouwer.schrijver} and P.~Cameron--D.~Fon-der-Flaass \cite{fonderflaass}, a strengthening of a theorem of P.~Cameron--D.~Fon-der-Flaass \cite{fonderflaass}, and several new results on the order of $K$-promotion.
Finally, we conjecture the order of rowmotion on plane partitions of height $3$ (which we have shown to be also the order of $K$-promotion on certain classes of increasing tableaux). In Section~\ref{subsec:conj}, we give another example of resonance on \emph{fully-packed loop configurations} and propose  additional instances of resonance related to \emph{alternating sign matrices} and \emph{totally symmetric self-complementary plane partitions}.

\section{$K$-Promotion on increasing tableaux}
\label{sec:increasing_tableaux}
In this section, we study increasing tableaux, the first of the objects in our main bijection (Theorem~\ref{thm:mainbij}). After recalling the basic concepts, we establish resonance of increasing tableaux under $K$-promotion in Theorem~\ref{cor:resonance_of_increasing_tableaux}. In Section~\ref{sec:K-BK}, we reinterpret $K$-promotion in terms of $K$-Bender-Knuth involutions, which we introduce; this interpretation plays an important role in Section~\ref{sec:equivariant_bijection} in establishing equivariance of our main bijection. In Section~\ref{sec:descent_cycling}, we extend a \emph{descent cycling} result of B.~Rhoades \cite[Lemma~3.3]{rhoades:thesis} from standard Young tableaux to increasing tableaux; this extension is used in Theorem~\ref{thm:our_cfdf} to improve on a theorem of P.~Cameron--D.~Fon-der-Flaass \cite{fonderflaass}.

\subsection{Increasing tableaux}
\label{sec:incdef}
Identify a partition $\lambda$ with its Young diagram. (Throughout this paper, we use the English orientation on Young diagrams.) An {\bf increasing tableau} of shape $\lambda$ is a filling of $\lambda$ with positive integers such that labels strictly increase from left to right across rows and from top to bottom down columns. An example appears in Figure~\ref{fig:an_increasing_tableau}. We write $\inc{\lambda}{q}$ for the set of all increasing tableaux of shape $\lambda$ with all entries at most $q$. (In contrast to other definitions that have appeared in the literature, we do not assume here that every integer between $1$ and $q$ appears.)

\begin{figure}[h]
\ytableaushort{1458,2579,679{10},8{10}}
\caption{An increasing tableau $T$ of shape $\lambda = (4,4,4,2)$.}
\label{fig:an_increasing_tableau}
\end{figure}

Increasing tableaux have appeared in various contexts within algebraic combinatorics. Most notably for our purposes, H.~Thomas--A.~Yong introduced \cite{thomas.yong:V} a \emph{$K$-jeu de taquin} algorithm for increasing tableaux, which they applied to \emph{$K$-theoretic Schubert calculus}, obtaining Littlewood-Richardson rules for the Grothendieck rings of algebraic vector bundles over \emph{Grassmannians}. This algorithm has been has been extended to the $K$-theory of a wider variety of spaces by \cite{buch.ravikumar, clifford.thomas.yong, buch.samuel}, as well as to the \emph{torus-equivariant} $K$-theory of Grassmannians \cite{thomas.yong:H_T, pechenik.yong:long}.

In \cite{pechenik}, the second author studied a $K$-promotion operator, analogous to that of M.-P.~Sch\"{u}tzenberger for semistandard tableaux \cite{schutzenberger}, but using $K$-\emph{jeu de taquin} in place of ordinary \emph{jeu de taquin}. $K$-promotion has been further studied by J.~Bloom--D.~Saracino and the second author \cite{bloom.pechenik.saracino}, T.~Pressey--A.~Stokke--T.~Visentin \cite{pressey.stokke.visentin} and B.~Rhoades \cite{rhoades}.

{\bf $K$-promotion} is defined as follows. Let $T \in \inc{\lambda}{q}$. Delete all labels $1$ from $T$. (Note there is at most one such label.) Consider the set of boxes that are either empty or contain $2$. This set naturally decomposes into connected components that are {\bf short ribbons}, i.e.\ connected sets of boxes containing no $2 \times 2$ subshape and with each column and row of length at most $2$. For each such short ribbon containing more than one box, we delete each label $2$, while simultaneously placing $2$ in each empty box. We do not make any change to short ribbons consisting of a single box. Now consider the set of boxes that are either empty or contain $3$, and repeat the above process. Continue until all empty boxes are located at outer corners of $\lambda$. Finally, label those boxes $q+1$ and then subtract $1$ from each entry. The result is $\kpro(T) \in \inc{\lambda}{q}$ (see Figure~\ref{fig:Kpromotion}).

\begin{figure}[h]
\begin{tikzpicture}
\ytableausetup{boxsize=1.2em}
\node (A) {$T=\ytableaushort{1246,4567}$};
\node[right=1.5 of A] (B) {\ytableaushort{{*(Dandelion)\blank} {*(Dandelion)2}46, 4567}};
\node[right=1.5 of B] (C) {\ytableaushort{2 {*(Dandelion)\blank} 46, 4567}};
\node[below=1 of C] (D) {\ytableaushort{2 {*(Dandelion)\blank} {*(Dandelion) 4} 6, {*(Dandelion)4} 567}};
\node[left=1.5 of D] (E) {\ytableaushort{2 4 {*(Dandelion)\blank} 6, 4 {*(Dandelion)5} 67}};
\node[left=1.5 of E] (F) {\ytableaushort{2 4 {*(Dandelion)\blank} {*(Dandelion)6}, 4 5 {*(Dandelion)6} 7}};
\node[below=1 of F] (G) {\ytableaushort{2 4 6{*(Dandelion)\blank}, 4 5 {*(Dandelion) \blank} {*(Dandelion) 7}}};
\node[right=1.5 of G] (H) {\ytableaushort{2 4 6 7, 4 5 7 {*(Dandelion) \blank} }};
\node[right=1.5 of H] (I) {$\ytableaushort{1 3 5 6, 3 4 6 7 }=\kpro(T)$};
\path (A) edge[pil] node[above]{Delete $1$'s} (B);
\path (B) edge[pil] (C);
\path (C) edge[pil] (D);
\path (D) edge[pil] (E);
\path (E) edge[pil] (F);
\path (F) edge[pil] (G);
\path (G) edge[pil] (H);
\path (H) edge[pil] node[above]{Fill and} node[below]{decrement} (I);
\end{tikzpicture}
\caption{Calculating the $K$-promotion of $T \in \inc{2\times 4}{7}$. In each intermediate step, we have colored the short ribbons on which we are about to act.}
\label{fig:Kpromotion}
\end{figure}
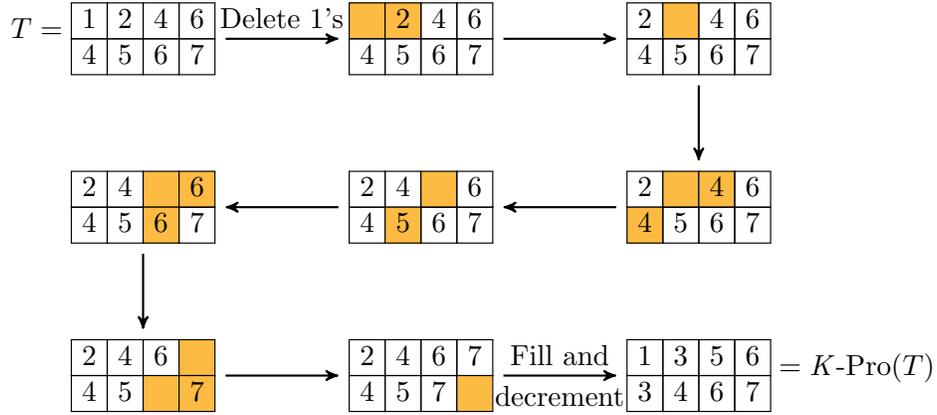

\subsection{Binary content cycling}
Define the {\bf binary content} of an increasing tableau $T \in \inc{\lambda}{q}$ to be the sequence ${\rm Con}(T) =(a_1, a_2, \dots, a_q)$, where $a_i = 1$ if $i$ is an entry of $T$ and $a_i = 0$ if it is not. That is, $a_i := \chi_i(T)$ where $\chi_i$ denotes the indicator function for the label $i$.

\begin{lemma}\label{prop:content_cycling}
Let $T \in \inc{\lambda}{q}$. If ${\rm Con}(T)=(a_1, a_2, \dots, a_q)$, then \linebreak ${\rm Con}(\kpro(T))$ is the cyclic shift $(a_2, \dots, a_q, a_1)$.
\end{lemma}
\begin{proof}
{\sf Case 1: ($\chi_1(T) := a_1 = 0$):} Then $T$ has no labels $1$. Hence the first step of $K$-promotion is trivial, deleting no labels. The ribbon switching process is also trivial, as there are no empty boxes. Therefore, at the final step, there are no boxes to fill. Thus, in this case, the total effect of $K$-promotion is merely to subtract $1$ from each entry. The lemma is then immediate in this case.

{\sf Case 2: ($\chi_1(T) := a_1 = 1$):} Then the first step of $K$-promotion is to delete a nonempty collection of labels $1$. Hence there are a nonzero number of empty boxes. The ribbon switching process may change the number of empty boxes, but clearly preserves its nonzeroness. Hence in the final step of $K$-promotion, there will be a nonzero number of boxes filled with $q+1$ and then decremented by $1$. Hence $\chi_q(\kpro(T)) = 1$.

Let $i>1$ and suppose $\chi_i(T) = 1$. Then $i$ appears as an entry of $T$. The ribbon switching process preserves this property (though not in general the number of entries $i$). Hence after subtracting one from each entry, this yields $\chi_{i-1}(\kpro(T)) = 1$. If instead $\chi_i(T) = 0$, then $i$ does not appear in $T$. Hence during the ribbon switching process, when we consider the ribbons consisting of $i$'s and empty boxes, each is a single empty box and by definition we make no change. Hence the ribbon switching process preserves the absence of $i$. After decrementing, this yields $\chi_{i-1}(\kpro(T)) = 0$.
\end{proof}

The following instance of resonance follows directly from Lemma~\ref{prop:content_cycling}.
\begin{theorem}
\label{cor:resonance_of_increasing_tableaux}
$(\inc{\lambda}{q}, \langle \kpro \rangle, {\rm Con})$ exhibits resonance with frequency~$q$.
\end{theorem}

Together with the following fact, this leads to a useful corollary.

\begin{fact}\label{fact:referee}
If $p$ is prime and $(X,G,f)$ exhibits resonance with frequency $p$, then for any $x \in X$ with $c \cdot f(x) \neq f(x)$, the $G$-orbit of $x$ has cardinality a multiple of $p$.
\end{fact}

\begin{corollary}
\label{cor:weakest_CFdFlike_statement}
Suppose $q$ is prime and $T \in \inc{\lambda}{q}$ does not have full binary content. Then the size of the $K$-promotion orbit of $T$ is a multiple of $q$.
\end{corollary}
\begin{proof}
By Theorem~\ref{cor:resonance_of_increasing_tableaux} and Fact~\ref{fact:referee}.
\end{proof}

\subsection{$K$-Bender-Knuth involutions}
\label{sec:K-BK}

In this subsection, we reinterpret $K$-promotion as a product of involutions, which we will need in our proof of Theorem~\ref{cor:rowinc}.
We define operators $\kbk_i$ on $\inc{\lambda}{q}$ for each $1 \leq i \leq q$. Take $T \in \inc{\lambda}{q}$. We compute $\kbk_i(T)$ as follows: Consider the set of boxes in $T$ that contain either $i$ or $i+1$. This set decomposes into connected components that are short ribbons. On each nontrivial such component, we do nothing. On each component that is a single box, replace the symbol $i$ by $i+1$ or \emph{vice versa}. The result is $\kbk_i(T)$. That is, the action of $\kbk_i$ on $T$ is to increment $i$ and/or decrement $i+1$, \emph{wherever possible}. These operators are illustrated in Figure~\ref{fig:KBenderKnuth}.

Clearly each $\kbk_i$ is an involution. We call it the $i$th {\bf $K$-Bender-Knuth involution} because in the case $T$ is standard, $\kbk_i$ coincides with the classical involution introduced by E.~Bender--D.~Knuth \cite{bender.knuth}.

\begin{figure}[ht]
\begin{tikzpicture}
\node (A) {\ytableaushort{1{*(Dandelion) 4}5{*(SkyBlue) 8},257{*(SkyBlue) 9},67{*(SkyBlue) 9}{10},{*(SkyBlue) 8}{10}}};
\node[below left = 2 and 2 of A] (B) {\ytableaushort{1{*(Dandelion) 3}58,2579,679{10},8{10}}};
\node[below right = 2 and 2 of A] (C) {\ytableaushort{145{*(SkyBlue) 8},257{*(SkyBlue) 9},67{*(SkyBlue) 8}{10},{*(SkyBlue) 9}{10}}};
\path (A) edge[pil, color=Dandelion] node[left]{$\kbk_3$} (B);
\path (A) edge[pil, color=SkyBlue] node[right]{$\kbk_8$} (C);
\end{tikzpicture}
\caption{The action of some $K$-Bender-Knuth involutions on the tableau $T$ from Figure~\ref{fig:an_increasing_tableau}.}
\label{fig:KBenderKnuth}
\end{figure}
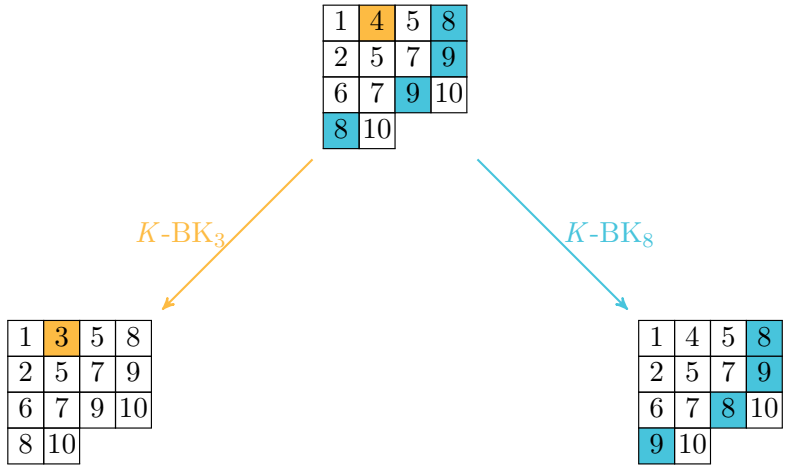

\begin{proposition}
\label{prop:kbk}
For $T \in \inc{\lambda}{q}$, $\kpro(T) = \kbk_{q-1} \circ \dots \circ \kbk_1 (T)$.
\end{proposition}
\begin{proof}
Another way to think of $\kbk_i$ is as the \emph{$K$-infusion} \cite[Section~3]{thomas.yong:V} of the labels $i$ through the labels $i+1$. That is, treat the labels $i$ as empty boxes and swap the short ribbons of empty boxes and $(i+1)$'s as in the definition of $K$-promotion; then relabel each $i+1$ as $i$ and each empty box as $i+1$.

From this characterization, it is clear that $\kbk_{q-1} \circ \dots \circ \kbk_1$ amounts to deleting the $1$'s and swapping the empty boxes successively through each other label in order, decrementing each other label as the empty boxes swap through it, and finally labeling the empty boxes at outer corners by $q$. This is transparently the same as $K$-promotion, except that the decrementing of labels happens throughout the process instead of all at the end.
\end{proof}

\subsection{Descent cycling}
\label{sec:descent_cycling}

In this subsection, we restrict consideration to increasing tableaux of rectangular shape. We extend a result of B.~Rhoades \cite[Lemma 3.3]{rhoades:thesis} from standard Young tableaux to increasing tableaux. Our proof is a elaboration of his argument. We will use this result in Theorem~\ref{thm:our_cfdf} to improve on a theorem of P.~Cameron--D.~Fon-der-Flaass \cite[Theorem~6(a)]{fonderflaass}.  Throughout this section, we write ``East'', ``east'' and ``southEast'' to mean ``strictly east'', ``weakly east'' and
``weakly south and strictly east'' respectively, etc.

\begin{definition}
Let $T \in \inc{a \times b}{q}$. For $1 \leq i < q$, the symbol $i$ is a {\bf descent} of $T$ if some instance of $i$ appears in a higher row than some instance of $i+1$.
Additionally, $q$ is a {\bf descent} of $T$ if $q-1$ is a descent of $\kpro(T)$.
\end{definition}

\begin{lemma}\label{lem:descents}
Suppose $i$ is a descent of $T \in \inc{a \times b}{q}$. Then $i-1 \mod q$ is a descent of $\kpro(T)$.
\end{lemma}
\begin{proof}
Throughout this proof, we use the original definition of $K$-promotion involving empty boxes, instead of the $K$-Bender-Knuth alternative.

\noindent
{\bf Case 1: ($1 < i < q$):}
$T$ has an instance of $i$ in row $h$ and an instance of $i+1$ in row $k$ with $h < k$. In $\kpro(T)$, there is an $i-1$ in row $h$ or $h-1$ and there is an $i$ in row $k$ or $k-1$. Hence $i-1$ is a descent in $\kpro(T)$ if $k-h >1$. Thus assume $k = h+1$.

Restrict attention to rows $h$ and $h+1$ of $T$. $T$ has a unique $i$ in row $h$ and a unique $i+1$ in row $h+1$. By increasingness, this $i+1$ is not East of this $i$.

Suppose the $i+1$ is West of the $i$. Then $T$ contains the local configuration
\scalebox{.6}{\ytableausetup{boxsize=2.2em}
$\ytableaushort{y z, {i \!+ \!1}}$}. Since $z \leq i < i+1$, the $i+1$ cannot move North during this application of $K$-promotion. Hence $\kpro(T)$ has $i$ in row $h+1$, and $i-1$ is a descent of $\kpro(T)$.

Thus, it remains to consider the case that $i$ and $i+1$ are in the same column of $T$. The $i+1$ can only move North if the $i$ moves. If the $i$ moves North, we are done, so assume $i$ moves West. Then $T$ has the local configuration \scalebox{.6}{$\ytableaushort{\none i, y {i \! +\! 1}}$} where $y \geq i$. But by increasingness, $y < i+1$. Hence $y=i$, so $T$ has the local configuration \scalebox{.6}{$\ytableaushort{\none i, i {i \! + \! 1}}$}. Therefore, $\kpro(T)$ has the local configuration \scalebox{.6}{$\ytableaushort{{i \!- \! 1} i, i}$} and thus $i-1$ is a descent of $\kpro(T)$.

\noindent
{\bf Case 2: ($i=1$):}
We must show that $q$ is a descent of $\kpro(T)$, that is, $q-1$ is a descent of $\kpro^2(T)$.

For $V \in \inc{a \times b}{q}$, let $\mathcal{F}(V)$ be the {\bf flow path} of $V$, that is the set of pairs of adjacent boxes $\{B, B'\}$ of $V$ such that $B$ and $B'$ are at some point part of the same short ribbon during the application of $\kpro$ to $V$. 
For $B$ a box of $a \times b$, we write $B^\uparrow$ for the box immediately North of $B$, $B^\rightarrow$ for the box immediately East of $B$, etc. Define the {\bf upper flow path} $\overline{\mathcal{F}}(V)$ to be those $\{B, B^\rightarrow\} \in \mathcal{F}(V)$ such that $\{B, B^\rightarrow\}$ is northmost in its columns among $\mathcal{F}(V)$ together with those $\{B, B^\downarrow\}$ such that $\{B, B^\downarrow\}$ is eastmost in its rows among $\mathcal{F}(V)$. Similarly define the {\bf lower flow path} $\underline{\mathcal{F}}(V)$ to be those pairs in $\mathcal{F}(V)$ that are southmost or westmost. Figure~\ref{fig:FlowPath} shows an example of these flow paths.

\begin{figure}[ht]
%\begin{picture}(100,130)
%\ytableausetup{boxsize=2.2em}
%\put(0,100){\ytableaushort{1248,2579,368{10},6{11}{13}{14},7{12}{15}{16}}}
%\put(1,2){\tikz \draw[BrickRed] \hellipse ;}
%\put(26,2){\tikz \draw[BrickRed] \hellipse ;}
%\put(52,2){\tikz \draw[BrickRed] \hellipse ;}
%\put(1,52){\tikz \draw[Dandelion] \hellipse ;}
%\put(26,52){\tikz \draw[Dandelion] \hellipse ;}
%\put(52,52){\tikz \draw[blue] \hellipse ;}
%\put(1,101){\tikz \draw[blue] \hellipse ;}
%\put(26,101){\tikz \draw[blue] \hellipse ;}
%\put(1,3){\tikz \draw[BrickRed] \vellipse ;}
%\put(1,27){\tikz \draw[BrickRed] \vellipse ;}
%\put(1,52){\tikz \draw[BrickRed] \vellipse ;}
%\put(1,77){\tikz \draw[BrickRed] \vellipse ;}
%\put(50,52){\tikz \draw[blue] \vellipse ;}
%\put(50,77){\tikz \draw[blue] \vellipse ;}
%\put(74,3){\tikz \draw[blue] \vellipse ;}
%\put(74,27){\tikz \draw[blue] \vellipse ;}
%
%\end{picture}
\begin{tikzpicture}
\draw (0,0) grid (4,5);
\node[] (A) at (0.5,4.5) {$1$};
\node[] (B) at (1.5,4.5) {$2$};
\node[] (C) at (2.5,4.5) {$4$};
\node[] (D) at (3.5,4.5) {$8$};
\node[] (E) at (0.5,3.5) {$2$};
\node[] (F) at (1.5,3.5) {$5$};
\node[] (G) at (2.5,3.5) {$7$};
\node[] (H) at (3.5,3.5) {$9$};
\node[] (I) at (0.5,2.5) {$3$};
\node[] (J) at (1.5,2.5) {$6$};
\node[] (K) at (2.5,2.5) {$8$};
\node[] (L) at (3.5,2.5) {$10$};
\node[] (M) at (0.5,1.5) {$6$};
\node[] (N) at (1.5,1.5) {$11$};
\node[] (O) at (2.5,1.5) {$13$};
\node[] (P) at (3.5,1.5) {$14$};
\node[] (Q) at (0.5,0.5) {$7$};
\node[] (R) at (1.5,0.5) {$12$};
\node[] (S) at (2.5,0.5) {$15$};
\node[] (T) at (3.5,0.5) {$16$};
\draw[<-,color=blue] (A) -- (B);
\draw[<-,color=blue] (B) -- (C);
\draw[<-,color=blue] (C) -- (G);
\draw[<-,color=blue] (G) -- (K);
\draw[<-,color=blue] (K) -- (L);
\draw[<-,color=blue] (L) -- (P);
\draw[<-,color=blue] (P) -- (T);
\draw[<-,color=BrickRed] (A) -- (E);
\draw[<-,color=BrickRed] (E) -- (I);
\draw[<-,color=BrickRed] (I) -- (M);
\draw[<-,color=BrickRed] (M) -- (Q);
\draw[<-,color=BrickRed] (Q) -- (R);
\draw[<-,color=BrickRed] (R) -- (S);
\draw[<-,color=BrickRed] (S) -- (T);
\draw[<-,color=Dandelion] (I) -- (J);
\draw[<-,color=Dandelion] (J) -- (K);
\end{tikzpicture}
\caption{The flow path of a tableau $V \in \inc{5 \times 4}{16}$. Elements of the lower flow path are shown in {\color{BrickRed} red}, while elements of the upper flow path are shown in {\color{blue} blue} and the remaining elements of the flow path are shown in {\color{Dandelion} yellow-orange}.}
\label{fig:FlowPath}
\end{figure}

Let $Q$ be the box in the lower right corner of $a \times b$.  By Proposition~\ref{prop:content_cycling}, $q$ appears in $\kpro(T)$. Hence by increasingness, $\kpro(T)$ has $q \in Q$. Thus it suffices to show that $\{ Q^\uparrow, Q\} \in \mathcal{F}(\kpro(T))$. The proof proceeds by comparing $\mathcal{F}(T)$ and $\mathcal{F}(\kpro(T))$.

Let $\overline{S} = \{ B \in a \times b : \{B, B^\rightarrow \} \in \overline{\mathcal{F}}(T)\}$. It is clear that $\overline{S}$ contains exactly one box from each column of $a \times b$, except the eastmost column.

If $\{ Q^\uparrow, Q\} \notin \mathcal{F}(\kpro(T))$, then there is some $B \in \bigcup \overline{\mathcal{F}}(\kpro(T))$ such that $B \in \overline{S}$. Choose $B$ to be maximally west among such boxes.

Since $B$ is chosen maximally west, $\{B^\leftarrow, B\} \notin \overline{\mathcal{F}}(\kpro(T))$. Suppose $\{B^\uparrow, B\} \in \overline{\mathcal{F}}(\kpro(T))$. Then in $\kpro(T)$, the entry of $B$ is strictly less than the entry of $B^{\uparrow\rightarrow}$. That is, if $h$ is the entry of $B$ and $k$ is the entry of $B^{\uparrow\rightarrow}$ then $h < k$. However, in $T$ we have $k+1 \in B^{\uparrow\rightarrow}$ and $h+1 \in B^\rightarrow$; this contradicts the increasingness of $T$. Thus $B$ is the northwestmost box of $a \times b$.

Since $1$ is a descent of $T$ and $B \in \overline{S}$, $T$ has $1 \in B$, $2 \in B^\downarrow$ and $2 \in B^\rightarrow$.

Let $\underline{S} = \{ B \in a \times b : \{B, B^\rightarrow \} \in \underline{\mathcal{F}}(T)\}$. 
We claim that if $\{B, B^\rightarrow \} \in \underline{\mathcal{F}}(T)$, then there is a pair $\{A, A^\rightarrow \} \in \mathcal{F}(\kpro(T))$ with $A$ North of $B$ in the same column. To see this, first observe by local analysis that if $\{B, B^\rightarrow \} \in \mathcal{F}(T)$ and $B^\uparrow \in \bigcup\mathcal{F}(\kpro(T))$, then $\{ B^\uparrow, B^{\uparrow\rightarrow}\} \in \mathcal{F}(\kpro(T))$. Now recall that $\underline{S}$ contains exactly one box from each column of $a \times b$, except the eastmost column. Moreover since $T$ has $2 \in B^\downarrow$, no box of $\underline{S}$ is in the northmost row. The claim follows.
Thus $Q^\uparrow \in \bigcup\mathcal{F}(\kpro(T))$ and we are done.

\noindent
{\bf Case 3: ($i=q$):}
By definition.
\end{proof}

\begin{proposition}\label{prop:descent_cycling}
The symbol $i$ is a descent of $T \in \inc{a \times b}{q}$ if and only if $i-1 \mod q$ is a descent of $\kpro(T)$.
\end{proposition}
\begin{proof}
Suppose $i$ is a descent of $T$. By Lemma~\ref{lem:descents}, $i-1 \mod q$ is a descent of $\kpro(T)$. Since $\inc{a \times b}{q}$ is finite, there is some $M$ such that $\kpro^M(T) = \kpro^{-1}(T)$. Hence by $M$ applications of Lemma~\ref{lem:descents}, $i+1$ is a descent of $\kpro^{-1}(T)$.
\end{proof}

\begin{definition}
Let $T \in \inc{a \times b}{q}$. For $1 \leq i < q$, $i$ is {\bf transpose descent} of $T$ if some instance of $i$ appears in a lower indexed column than some instance of $i+1$. Additionally $q$ is a {\bf transpose descent} of $T$ if $q-1$ is a descent of $\kpro(T)$.

Equivalently, $j$ is a transpose descent of $T$ if and only if $j$ is a descent of the transpose of $T$.
\end{definition}

\begin{proposition}\label{prop:transpose_descent_cycling}
The symbol $i$ is a transpose descent of $T \in \inc{a \times b}{q}$ if and only if $i-1 \mod q$ is a transpose descent of $\kpro(T)$.
\end{proposition}
\begin{proof}
Since clearly $K$-promotion commutes with transposing, the proposition is
immediate from Proposition~\ref{prop:descent_cycling}.
\end{proof}

The following is an enriched version of Corollary~\ref{cor:weakest_CFdFlike_statement} for rectangular tableaux.

\begin{proposition}
\label{prop:content_descents}
Let $T \in \inc{a \times b}{q}$ with $q$ prime. Suppose at least one of the following is true:
\begin{itemize}
\item $T$ does not have full binary content,
\item some $1 \leq i \leq q$ is not a descent in $T$, or 
\item some $1 \leq i \leq q$ is not a transpose descent in $T$.
\end{itemize}
 Then, the $K$-promotion orbit of $T$ has cardinality a multiple of $q$.
\end{proposition}
\begin{proof}
If $T$ does not have full binary content, Corollary~\ref{cor:weakest_CFdFlike_statement} applies. Otherwise, some $1 \leq i \leq q$ is not a (transpose) descent in $T$. The proposition is then immediate by Fact~\ref{fact:referee} together with either Proposition~\ref{prop:descent_cycling} or~\ref{prop:transpose_descent_cycling}.
\end{proof}

Finally, we prove the following lemma, which we will use in Section~\ref{subsec:conseq}.

\begin{lemma}\label{lem:three_for_two}
Let $T \in \inc{a \times b}{q}$ and suppose that $1 \leq i < q$ is both a descent and a transpose descent in $T$. Then the number of $i$'s in $T$ plus the number of $(i+1)$'s in $T$ is at least $3$.
\end{lemma}
\begin{proof}
Since $i$ is a descent, both $i$ and $i+1$ must appear in $T$. Hence if $i$ appears at least twice in $T$, we are done. Thus assume $i$ appears exactly once in $T$. Since $i$ is a descent, some $i+1$ appears South of this $i$. Since $i$ is a transpose descent, some $i+1$ appears East of this $i$.

We claim these instances of $i+1$ are distinct, completing the proof of the lemma. Otherwise, we have $i+1$ SouthEast of $i$. Consider the label $z$ of the box that is in the row of the $i$ and in the column of the $i+1$. By the increasingness conditions on $T$, $i < z < i+1$, contradicting that $z$ is an integer.
\end{proof}

In Section~\ref{subsec:conseq}, we will use Proposition~\ref{prop:content_descents}, Lemma~\ref{lem:three_for_two}, and our main results, Theorems~\ref{thm:3Dprorow} and \ref{cor:rowinc}, to give in Theorem~\ref{thm:our_cfdf} a strengthening of a theorem of P.~Cameron--D.~Fon-der-Flaass on plane partitions in a box. 

\section{Promotion and rowmotion, revisited}\label{sec:toggle}

In this section, we switch our focus from increasing tableaux to our other main objects of study: \emph{plane partitions}. A \textbf{plane partition} is a stack of unit cubes in the positive orthant, justified toward the origin in all three directions. Plane partitions inside a box with side lengths $a$, $b$, and $c$, are counted by P.~MacMahon's box formula: $\displaystyle\prod\displaystyle\frac{i+j+k-1}{i+j+k-2}$ where the product is over all $1\leq i\leq a$, $1\leq j\leq b$, $1\leq k\leq c$~\cite{MacMahon}.

Plane partitions inside an $a\times b\times c$ box can be seen as \emph{order ideals} in the product of three chains poset $\mathbf{a}\times \mathbf{b}\times \mathbf{c}$. Thus, most of our discussion in this section centers on posets and order ideals, keeping in mind that all such general results can be applied to plane partitions.

We begin in Section~\ref{subsec:row} by discussing the \emph{rowmotion} action on order ideals and some results on the order of this action on products of two and three chains. In Section~\ref{subsec:toggle}, we discuss the 
\emph{toggle group}, first defined by P.~Cameron--D.~Fon-der-Flaass~\cite{fonderflaass} and further studied by N.~Williams and the third author~\cite{prorow}. In Section~\ref{sec:resonance_pp},  we use the main theorem of~\cite{prorow} to prove resonance of plane partitions under rowmotion.
The toggle group will be the algebraic structure underlying Sections~\ref{subsec:3Dprorow} and~\ref{subsec:provpi}, in which we revisit this main result of~\cite{prorow} by proving, in Theorem~\ref{thm:3Dprorow}, a generalization in the setting of \emph{$n$-dimensional lattice projections}. 

\subsection{Rowmotion}
\label{subsec:row}
Let $P$ be a finite partially ordered set (poset). 
$P$ is a \textbf{chain} if all its elements are mutually comparable. Let $\mathbf{n}$ denote the $n$-element chain.
The \textbf{product of $k$ chains} poset, $P=\mathbf{n_1}\times\mathbf{n_2}\times\cdots \mathbf{n_k}$, has as elements ordered integer $k$-tuples  $(x_1,x_2,\ldots,x_k)$ such that $0\leq x_i\leq n_i-1$ with partial order given by componentwise comparison.

A subset $I \subseteq P$ is an {\bf order ideal} if it is closed downward, i.e.\ if $y \in I$ and $x \leq y$, then $x \in I$. Denote the set of order ideals of $P$ as $J(P)$. An order ideal in $P$ is uniquely determined by its set of maximal elements, or alternatively by the set of minimal elements of its complement in $P$. We study the orbit structure of \textbf{rowmotion}, $\row \colon J(P) \to J(P)$, defined as the order ideal whose maximal elements are the minimal elements of $P \setminus I$.

The function $\row$ has a long history of rediscovery and has appeared under many names. A partial summary of previous work follows; for a more complete discussion, see~\cite{prorow}. A.~Brouwer--A.~Schrijver \cite{brouwer.schrijver} studied $\row$ for $P = {\bf a} \times {\bf b}$, the product of two chains. They discovered that this action has much smaller orbits than one naively expects:
\begin{theorem}[{\cite[Theorem~3.6]{brouwer.schrijver}}]
\label{thm:square}
The order of $\row$ on $J({\bf a} \times {\bf b})$ is $a + b$.
\end{theorem}

P.~Cameron--D.~Fon-der-Flaass \cite{fonderflaass} studied the same question on plane partitions, that is, $J({\bf a} \times {\bf b} \times {\bf c})$. 
\begin{theorem}[{\cite[Theorem~6(b)]{fonderflaass}}]
\label{thm:height2}
The order of $\row$ on $J({\bf a} \times {\bf b} \times {\bf 2})$ is $a + b + 1$.
\end{theorem}

Extrapolating from Theorems~\ref{thm:square} and~\ref{thm:height2}, one might speculate that $\row$ has order $a+b+c-1$ on $J({\bf a} \times {\bf b} \times {\bf c})$. In general, the order is unknown but often significantly greater than this naive guess. However, P.~Cameron--D.~Fon-der-Flaass established the following related fact. 
\begin{theorem}[{\cite[Theorem~6(a)]{fonderflaass}}]
\label{thm:cfdf}
If $a + b + c -1$ is prime and $c > ab - a - b + 1$, then the cardinality of every orbit of $\row$ on $J({\bf a} \times {\bf b} \times {\bf c})$ is a multiple of $a+b+c-1$.
\end{theorem}

We will revisit Theorems~\ref{thm:square} and \ref{thm:height2} in Remark~\ref{remark:p_pp}.
In Section~\ref{sec:resonance_pp}, we give a new proof of Theorem~\ref{thm:cfdf}.
Furthermore, as a consequence of our main equivariant bijection between plane partitions and increasing tableaux (Theorem~\ref{cor:rowinc}), we will show, in Theorem~\ref{thm:our_cfdf}, that in Theorem~\ref{thm:cfdf} the condition $c > ab - a -b +1$ may be relaxed to $c > \frac{2ab-2}{3} - a - b + 2$. This is evidence toward the conjecture of P.~Cameron--D.~Fon-der-Flaass \cite{fonderflaass} that this condition may be dropped entirely.

The approach of P.~Cameron--D.~Fon-der-Flaass was to reinterpret rowmotion as a \emph{toggle group action}. We describe the toggle group in the next subsection. 
%the following subsections.

\subsection{The toggle group}
\label{subsec:toggle}

The \textbf{toggle group} was first studied by P.\ Cameron--D.~Fon-der-Flaass~\cite{fonderflaass} and subsequently N.~Williams and the third author \cite{prorow}. It is the subgroup of the symmetric group on all order ideals $\mathfrak{S}_{J(P)}$ generated by certain involutions, called \textbf{toggles}.
%\begin{definition}
For each element $e\in P$ define its \textbf{toggle} $t_e:J(P)\rightarrow J(P)$ as follows.
\[
 t_e(X) =
  \begin{cases}
   X\cup\{e\} & \text{if } e\notin X \text{ and } X\cup\{e\}\in J(P) \\
   X\setminus\{e\} & \text{if } e\in X \text{ and } X\setminus\{e\}\in J(P) \\
   X       & \text{otherwise}
  \end{cases}
\]
%\end{definition}

%\begin{definition}[\cite{fonderflaass}]
%The \textbf{toggle group} $T(J(P))$ is the subgroup of the symmetric group $\mathfrak{S}_{J(P)}$ generated by $\{ t_e\}_{e \in P}$.
%\end{definition}

\begin{remark}
\label{remark:commute}
Observe that $t_e, t_f$ commute whenever neither $e$ nor $f$ covers the other.
\end{remark}

\medskip
The following theorem interprets rowmotion as a toggle group action.
\begin{theorem}[\cite{fonderflaass}]
\label{thm:row_tog}
Given any poset $P$, $\row$ is the toggle group element that toggles the elements of $P$ in the reverse order of any linear extension. If $P$ is ranked, this is the same as toggling the ranks (rows) from top to bottom.
\end{theorem}

In 2012~\cite{prorow}, N.~Williams and the third author built on the work of P.~Cameron--D.~Fon-der-Flaass, showing that rowmotion is conjugate to the toggle group action they called \textbf{promotion}, or $\pro$, defined as toggling the elements of the poset from left to right (given a suitable notion of left-to-right, for which they used the term \emph{rc-poset}). 

\begin{theorem}[{\cite[Theorem~5.2]{prorow}}]
\label{thm:StWiProrow}
For any rc-poset $P$, there is an equivariant bijection
between $J(P)$ under $\pro$ and $J(P)$ under $\row$.
\end{theorem}

We discuss this result in further detail in Sections~\ref{subsec:3Dprorow} and~\ref{subsec:provpi} and give a multidimensional generalization of it in Theorem~\ref{thm:3Dprorow}.

\begin{remark}
\label{remark:p_pp}
For many posets, the orbit structure of promotion is easier to study than that of rowmotion. Thus 
%this equivariant bijection between promotion and rowmotion 
Theorem~\ref{thm:StWiProrow} yielded many results on the orbit structure of rowmotion by translating from the analogous result on promotion.
%We discuss this further in Remark~\ref{remark:p_pp}.
%The power of the equivariant bijection of Theorem~\ref{thm:3Dprorow} and its two-dimensional version found in \cite{prorow} lies in its ability to translate difficult questions about one action into easy questions about another.
Theorem~\ref{thm:StWiProrow} was applied in \cite{prorow} to give simple new proofs of Theorem~\ref{thm:square} of A.~Brouwer--A.~Schrijver and Theorem~\ref{thm:height2} of P.~Cameron--D.~Fon-der-Flaass (discussed in Section~\ref{subsec:row}), as well as easy proofs of the \emph{cyclic sieving phenomenon} of V.\ Reiner, D.\ Stanton, and D.~White~\cite{reiner.stanton.white} in these cases and a few others. 
\end{remark}

In the next subsection, we use Theorem~\ref{thm:StWiProrow} 
%applied to plane partitions 
to prove resonance of rowmotion on plane partitions.

\subsection{Resonance of plane partitions}
\label{sec:resonance_pp}
In this subsection, we prove our second  resonance result, Theorem~\ref{cor:resonance_of_pp_directly_row}. We also give a new proof of Theorem~\ref{thm:cfdf}.

In \cite[Section~7.2]{prorow}, N.~Williams and the third author applied their theory to plane partitions, that is, the order ideals $J(\mathbf{a}\times\mathbf{b}\times\mathbf{c})$. They characterized 
%plane partitions in an $a\times b\times c$ box 
$J(\mathbf{a}\times\mathbf{b}\times\mathbf{c})$ in terms of \emph{boundary path matrices}. We give a sketch of this 
characterization here; for futher details, see~\cite{prorow}.
Given an order ideal in a special kind of planar poset (in the language of~\cite{prorow}, an rc-poset of height 1, or in the language of the next section, a poset with a 2-dimensional lattice projection), its \textbf{boundary path} is a binary sequence that encodes the path that separates the order ideal from the rest of the poset.
Given a plane partition $I\in J(\mathbf{a}\times\mathbf{b}\times\mathbf{c})$, its \textbf{boundary path matrix} is a $b\times (a+b+c-1)$ matrix $\{X_{i,j}\}$ with entries in $\{0,1\}$ such that
%where $1\leq i\leq c$ and $1\leq j\leq a+b+c-1$, and such that
%with $c$ rows and $a+b+c-1$ columns, 
the $i$th row consists of the boundary path of layer $i$ preceded by $i - 1$ zeros and succeeded by $b - i$ zeros. 
The rows of a boundary path matrix each sum to $a$ and the entries obey the condition 
\[ \mbox{if } \sum_{j=1}^k X_{i,j} = \sum_{j=1}^k X_{i+1,j},
 \mbox{ then } X_{i+1,j+1}\neq 1.\] It was noted in \cite[Section~7.2]{prorow} that $\pro$ traces from left to right through the columns of the boundary path matrix, swapping
each pair of entries in adjacent columns and the same row that result in a matrix still satisfying the condition above. 

Given $I \in J({\bf a} \times {\bf b} \times {\bf c})$ with boundary path matrix $\{X_{i,j}\}$, define $X_{\max}(I)$ to be the vector of length $a+b+c-1$ whose $j$th entry is $\max(X_{i,j})_{1\leq i\leq b}$.

\begin{lemma}
\label{prop:zeroscol_cycling}
Let $I \in J({\bf a} \times {\bf b} \times {\bf c})$. If $X_{\max}(I)=(x_1, x_2, \dots, x_{a+b+c-1})$, then $X_{\max}(\pro(I))$ is the cyclic shift $(x_2, \dots, x_{a+b+c-1}, x_1)$.
\end{lemma}
\begin{proof}
For $i>1$, if column $i$ of the boundary path matrix is all zeros, then in the application of $\pro$, all of these entries swap with the entries of column $i-1$, since the condition on the partial row sums is not violated.

If $i=1$, the column of all zeros swaps all the way through the matrix, from the first column to the last column.

Thus, under $\pro$, a column of all zeros cyclically shifts to the left.
\end{proof}

The following instance of resonance follows directly from Lemma~\ref{prop:zeroscol_cycling}.
\begin{proposition}
\label{cor:resonance_of_pp_directly} 
$(J(\mathbf{a}\times\mathbf{b}\times\mathbf{c}), \langle \pro \rangle, X_{\max})$ exhibits resonance with frequency~$a+b+c-1$.
\end{proposition}

 Let $D$ be the conjugating toggle group element between rowmotion and promotion given in \cite[Theorem~5.4]{prorow}. By the equivariance of $\pro$ and $\row$ in~\cite{prorow}, we have the following statement of resonance on rowmotion, whose proof follows directly from Proposition~\ref{cor:resonance_of_pp_directly} and \cite[Theorem~5.4]{prorow}.

\begin{theorem}
\label{cor:resonance_of_pp_directly_row} 
$(J(\mathbf{a}\times\mathbf{b}\times\mathbf{c}), \langle \row \rangle, X_{\max}\circ D)$ exhibits resonance with frequency~$a+b+c-1$.
\end{theorem}
%\begin{proof}
%This follows from Proposition~\ref{cor:resonance_of_pp_directly} and Theorem 5.4 of~\ref{prorow}.
%\end{proof}

This leads to the following corollary.
\begin{corollary}
\label{cor:weakest_CFdFlike_statement_pp}
Suppose $a+b+c-1$ is prime and $I \in J(\mathbf{a}\times\mathbf{b}\times\mathbf{c})$. Suppose there is a zero in $X_{\max}(I)$. Then the size of the promotion orbit of $I$ is a multiple of $a+b+c-1$.
\end{corollary}
\begin{proof}
By Theorem~\ref{cor:resonance_of_pp_directly_row} and Fact~\ref{fact:referee}.
\end{proof}

Using Corollary~\ref{cor:weakest_CFdFlike_statement_pp}, we have a new proof of Theorem~\ref{thm:cfdf} of P.~Cameron--D.~Fon-der-Flaass.

\begin{proof}[Proof of Theorem~\ref{thm:cfdf}] If  $a+b+c-1$ is prime and $c> a b - a - b + 1$, then there are a total of $a b$ ones in the boundary path matrix, but a total of $a+b+c-1 > a b$ columns in the matrix, so there must be a column of all zeros. Thus, there is a zero in $X_{\max}(I)$ for any plane partition $I$ in the $a\times b\times c$ box, and the promotion orbit is a multiple of $a+b+c-1$ by Corollary~\ref{cor:weakest_CFdFlike_statement_pp}. Then by Theorem~\ref{thm:StWiProrow}, the orbits of rowmotion are also multiples of $a+b+c-1$.
\end{proof}

 P.~Cameron--D.~Fon-der-Flaass's original proof of Theorem~\ref{thm:cfdf} is similar, though more complicated, since it analyzes rowmotion directly rather than conjugating to promotion. 

\subsection{$n$-dimensional lattice projections}
\label{subsec:3Dprorow}

In this and the next 
subsections, we adapt the proof of the conjugacy of promotion and rowmotion from~\cite{prorow} to give a generalization in the setting of \emph{$n$-dimensional lattice projections}, which we introduce in Definition~\ref{def:lattice_proj}. (This new perspective includes the original theorem as the case $n=2$.)
%For ease of notation, we now consider projections into  $\mathbb{Z}^n$ (rather than certain integer spans).
We prove, in Theorem~\ref{thm:3Dprorow}, the equivariance of the $2^{n-1}$ toggle group actions given in Definition~\ref{def:3DPro}.

%\begin{definition}
%We say that a map $f \colon P \to Q$ between two posets is \textbf{cover relation preserving} if $p_1 \lessdot_P p_2$ implies $f(p_1) \lessdot_Q f(p_2)$.
%\end{definition}

\begin{definition}
We say that a poset $P$ is {\bf ranked} if it admits a rank function $\rk \colon P \to \mathbb{Z}$ satsifying $\rk(y)=\rk(x)+1$ when $y$ covers $x$.
\end{definition}

\begin{definition}
\label{def:lattice_proj}
We say that an \textbf{($n$-dimensional) lattice projection} of a ranked poset $P$ is an order and rank preserving map $\pi:P\rightarrow \mathbb{Z}^n$, where the rank function on $\mathbb{Z}^n$ is the sum of the coordinates and $x\leq y$ in $\mathbb{Z}^n$ if and only if the componentwise difference $y-x$ is in $(\mathbb{Z}_{\geq 0})^n$.
\end{definition}

In light of Remark~\ref{remark:commute}, the key feature of $\pi$ is that it preserves cover relations. That is, if $y$ covers $x$ in $P$, then $\pi(y)$ covers $\pi(x)$ in $\mathbb{Z}^n$. However, since $\mathbb{Z}^n$ is ranked, $\pi$ being cover-relation preserving would make $\rk\circ\pi$ a rank function for $P$. And if $P$ is ranked, then a map $\pi:P\rightarrow \mathbb{Z}^n$ being cover-relation preserving is equivalent to it being order and rank preserving (up to a shift of the rank functions).

In~\cite{prorow}, the definition of an rc-poset was a poset that had a 2-dimensional lattice projection (albeit to a slightly different lattice). However, E.~Sawin noted that every ranked poset $P$ with rank function $\rho$ has such an embedding given by $\pi(x)=(\rho(x),0)$ for $x\in P$~\cite{esawin}. Similarly, any poset $P$ with a lattice projection $\pi$ has a rank function given by the sum of the coordinates in $\pi(x)$ for $x\in P$.

Additionally, a ranked poset may have multiple distinct projections. For example, in Figure~\ref{fig:cube}, we have the boolean lattice on three elements, which we may think of as a product of three chains of length 2. In Figure~\ref{fig:3dlattproj}, we have the standard three-dimensional lattice projection of this poset obtained by viewing it as a product of three chains. In Figure~\ref{fig:rcprojection}, we show two different two-dimensional lattice projections of this poset. In the projection on the right,  
we assign every element of the same rank to the same point, but instead of doing so along the $x$-axis as in the previous paragraph, we do this diagonally in a zig-zag pattern.
Therefore, instead of considering rc-posets, we consider any ranked poset, but with respect to a given lattice projection.

\begin{figure}[hbtp]
\begin{tikzpicture}
\node [below] at (0,0) {a};
\node [left] at (-1,1) {b};
\node [right] at (0,1) {c};
\node [right] at (1,1) {d};
\node [left] at (-1,2) {e};
\node [left] at (0,2) {f};
\node [right] at (1,2) {g};
\node [above] at (0,3) {h};
\draw (0,0) -- (1,1) -- (1,2) -- (0,3) -- (-1,2) -- (-1,1) -- (0,0);
\draw (-1,2) -- (0,1) -- (1,2);
\draw (-1,1) -- (0,2) -- (1,1);
\draw (0,0) -- (0,1);
\draw (0,3) -- (0,2);
\end{tikzpicture}
\caption{A product of three chains poset.}
\label{fig:cube}
\end{figure}
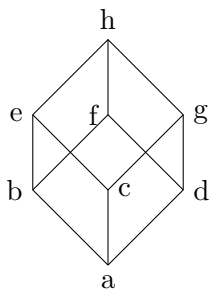

\begin{figure}[hbtp]
\begin{tikzpicture}[scale=1.75]
\node[inner sep=1pt,label=left:a,circle,draw,fill] at (0,0,0)  {};
\node[inner sep=1pt,label=left:b,circle,draw,fill] at (0,0,1) {};
\node[inner sep=1pt,label=left:c,circle,draw,fill] at (0,1,0) {};
\node[inner sep=1pt,label=above right:d,circle,draw,fill] at (1,0,0) {};
\node[inner sep=1pt,label=left:e,circle,draw,fill] at (0,1,1) {};
\node[inner sep=1pt,label=below:f,circle,draw,fill] at (1,0,1) {};
\node[inner sep=1pt,label=above:g,circle,draw,fill] at (1,1,0) {};
\node[inner sep=1pt,label=below left:h,circle,draw,fill] at (1,1,1) {};
\draw (0,0,0) -- (0,0,1) -- (0,1,1) -- (1,1,1);
\draw (0,0,0) -- (1,0,0) -- (1,1,0) -- (1,1,1);
\draw (0,0,0) -- (0,1,0) -- (1,1,0) -- (1,1,1);
\draw (1,0,0) -- (1,0,1) -- (1,0,0);
\draw (0,0,1) -- (1,0,1) -- (1,1,1);
\draw (0,1,0) -- (0,1,1);
\draw [dotted, ->] (0,0,1) -- (0,0,1.5) node[anchor=north] {$x$};
\draw [dotted, ->] (0,1,0) -- (0,1.5,0) node[anchor=west] {$y$};
\draw [dotted, ->] (1,0,0) -- (1.5,0,0) node [anchor=west] {$z$};
\end{tikzpicture}
\caption{The standard three-dimensional lattice projection of the poset of Figure~\ref{fig:cube}.}
%a product of three two-element chains.}
\label{fig:3dlattproj}
\end{figure}
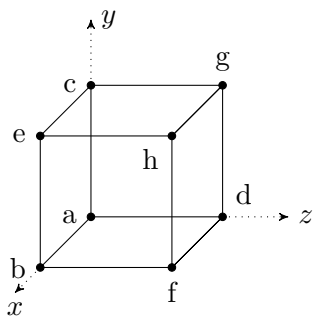

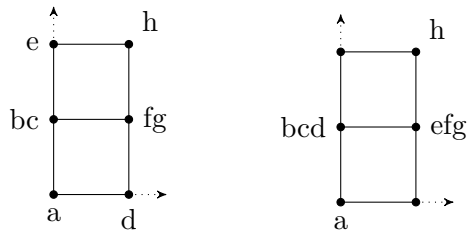
\begin{figure}[htbp]
\begin{tikzpicture}
\draw (0,0) grid (1,2);
\node[inner sep=1pt,circle,draw,fill,label=below: a] at (0,0) {};
\node[inner sep=1pt,circle,draw,fill,label=left: bc] at (0,1) {};
\node[inner sep=1pt,circle,draw,fill,label=below: d] at (1,0) {};
\node[inner sep=1pt,circle,draw,fill,label=right: fg] at (1,1) {};
\node[inner sep=1pt,circle,draw,fill,label=left: e] at (0,2) {};
\node[inner sep=1pt,circle,draw,fill,label=above right: h] at (1,2) {};
\draw (1,0) -- (1.5,0) [dotted,->];
\draw (0,2) -- (0,2.5) [dotted,->];
\end{tikzpicture}
\hspace{1cm}
\begin{tikzpicture}
\draw (0,0) grid (1,2);
\node[inner sep=1pt,circle,draw,fill,label=below: a] at (0,0) {};
\node[inner sep=1pt,circle,draw,fill,label=left: bcd] at (0,1) {};
\node[inner sep=1pt,circle,draw,fill] at (1,0) {};
\node[inner sep=1pt,circle,draw,fill,label=right: efg] at (1,1) {};
\node[inner sep=1pt,circle,draw,fill] at (0,2) {};
\node[inner sep=1pt,circle,draw,fill,label=above right: h] at (1,2) {};
\draw (1,0) -- (1.5,0) [dotted,->];
\draw (0,2) -- (0,2.5) [dotted,->];
\end{tikzpicture}
\caption{Two distinct two-dimensional lattice projections of the poset of Figure~\ref{fig:cube}.}
\label{fig:rcprojection}
\end{figure}

%\begin{figure}[htbp]
%\label{fig:gyrationprojection}
%\begin{tikzpicture}
%\draw (0,0) grid (1,2);
%\node[inner sep=1pt,circle,draw,fill,label=below: a] at (0,0) {};
%\node[inner sep=1pt,circle,draw,fill,label=left: bcd] at (0,1) {};
%\node[inner sep=1pt,circle,draw,fill] at (1,0) {};
%\node[inner sep=1pt,circle,draw,fill,label=right: efg] at (1,1) {};
%\node[inner sep=1pt,circle,draw,fill] at (0,2) {};
%\node[inner sep=1pt,circle,draw,fill,label=above right: h] at (1,2) {};
%\draw (1,0) -- (1.5,0) [dotted,->];
%\draw (0,2) -- (0,2.5) [dotted,->];
%\end{tikzpicture}
%\caption{Gyration lattice projections of a product of three chains.}
%\end{figure}

\subsection{Promotion via affine hyperplane toggles}
\label{subsec:provpi}

We now define a toggling order on our poset with respect an $n$-dimensional lattice projection, and with respect to a distinguished direction.

\begin{definition}
\label{def:3DPro}
Let $P$ be a poset with an $n$-dimensional lattice projection $\pi$, and let $v=(v_1,v_2,v_3,\ldots, v_{n})$, where $v_j\in\{\pm 1\}$. Let $T_{\pi,v}^i$ be the product of toggles $t_x$ for all elements $x$ of $P$ that lie on the affine hyperplane $\langle \pi(x),v\rangle=i$. If there is no such $x$, then this is the empty product, considered to be the identity. Then define \textbf{promotion with respect to $\pi$ and $v$} as the toggle product
$\pro_{\pi,v}=\ldots T_{\pi,v}^{-2}T_{\pi,v}^{-1}T_{\pi,v}^0T_{\pi,v}^{1}T_{\pi,v}^{2}\ldots$.
\end{definition}

See Figure~\ref{fig:hyperplanes} for an example.

\smallskip
\begin{remark}
Note that $\pro_{\pi,-v}=(\pro_{\pi,v})^{-1}$, so we will generally only consider distinguished vectors with $v_1=1$, as all promotion operators are either of this form, or the inverse of something of this form.
\end{remark}

\begin{lemma}
\label{lem:hyperplane}
Two elements of a poset that lie on the same affine hyperplane $\langle \pi(x),v\rangle=i$ cannot be part of a covering relation, so by Remark~\ref{remark:commute}, the operator $T_{\pi,v}^i$ is well-defined and $(T_{\pi,v}^i)^2=1$.
\end{lemma}

\begin{proof}
Assume that $y$ covers $x$, and they both lie on the same affine hyperplane ($\langle \pi(x),v\rangle=\langle \pi(y),v\rangle=i$). Then $\langle \pi(y),v\rangle - \langle \pi(x),v\rangle=\langle \pi(y)-\pi(x),v\rangle=0$. But since $y$ covers $x$, $\pi(y)-\pi(x)=e_i$ for some $i$. And since $v$ has all coordinates $\pm 1$, then $\langle e_i,v\rangle=\pm 1$, a contradiction.
\end{proof}

\begin{figure}[htbp]
\begin{tikzpicture}[scale=.8]
\begin{scope}[xshift=10cm,yshift=-7cm]
\planepartitionQ{{2,2}}
\draw[->,dotted] (0,0) -- ($1.5*(\xx,\xxx)$) node[anchor=north] {$y$};
\draw[->,dotted] (0,0) -- ($2.5*(\yy,\yyy)$) node[anchor=north] {$x$};
\draw[->,dotted] (0,0) -- ($2.75*(\zz,\zzz)$) node[anchor=west] {$z$};
%x+y-z=-2
 \coordinate (a1) at ($0*(\xx,\xxx)+0*(\yy,\yyy)+2*(\zz,\zzz)$);
 \node[fill=red,draw,circle,inner sep=.5ex] at (a1) {};
 \node at (0,-1) {$x+y-z=-2$};
\end{scope}

\begin{scope}[xshift=5cm,yshift=-7cm]
\planepartitionQ{{2,2}}
\draw[->,dotted] (0,0) -- ($1.5*(\xx,\xxx)$) node[anchor=north] {$y$};
\draw[->,dotted] (0,0) -- ($2.5*(\yy,\yyy)$) node[anchor=north] {$x$};
\draw[->,dotted] (0,0) -- ($2.75*(\zz,\zzz)$) node[anchor=west] {$z$};
%x+y-z=-1
\coordinate (a2) at ($0*(\xx,\xxx)+1*(\yy,\yyy)+2*(\zz,\zzz)$);
\coordinate (b2) at ($1*(\xx,\xxx)+0*(\yy,\yyy)+2*(\zz,\zzz)$);
\coordinate (c2) at ($0*(\xx,\xxx)+0*(\yy,\yyy)+1*(\zz,\zzz)$);
\node[fill=red,draw,circle,inner sep=.5ex] at (a2) {};
\node[fill=red,draw,circle,inner sep=.5ex] at (b2) {};
\node[fill=red,draw,circle,inner sep=.5ex] at (c2) {};
\draw[fill=light-gray,opacity=.4] (a2) -- (b2) -- (c2) -- cycle;
\node at (0,-1) {$x+y-z=-1$};
\end{scope}

\begin{scope}[xshift=0cm,yshift=-7cm]
\planepartitionQ{{2,2}}
\draw[->,dotted] (0,0) -- ($1.5*(\xx,\xxx)$) node[anchor=north] {$y$};
\draw[->,dotted] (0,0) -- ($2.5*(\yy,\yyy)$) node[anchor=north] {$x$};
\draw[->,dotted] (0,0) -- ($2.75*(\zz,\zzz)$) node[anchor=west] {$z$};
%x+y-z=0
\coordinate (a3) at ($0*(\xx,\xxx)+1*(\yy,\yyy)+1*(\zz,\zzz)$);
\coordinate (b3) at ($1*(\xx,\xxx)+0*(\yy,\yyy)+1*(\zz,\zzz)$);
\coordinate (c3) at ($0*(\xx,\xxx)+0*(\yy,\yyy)+0*(\zz,\zzz)$);
\coordinate (d3) at ($1*(\xx,\xxx)+1*(\yy,\yyy)+2*(\zz,\zzz)$);
\coordinate (e3) at ($0*(\xx,\xxx)+2*(\yy,\yyy)+2*(\zz,\zzz)$);
\node[fill=red,draw,circle,inner sep=.5ex] at (a3) {};
\node[fill=red,draw,circle,inner sep=.5ex] at (b3) {};
\node[fill=red,draw,circle,inner sep=.5ex] at (c3) {};
\node[fill=red,draw,circle,inner sep=.5ex] at (d3) {};
\node[fill=red,draw,circle,inner sep=.5ex] at (e3) {};
\draw[fill=light-gray,opacity=.4] (a3) -- (c3) -- (b3) -- (d3) -- (e3) -- cycle;
\node at (0,-1) {$x+y-z=0$};
\end{scope}

\begin{scope}[xshift=10cm,yshift=0cm]
\planepartitionQ{{2,2}}
\draw[->,dotted] (0,0) -- ($1.5*(\xx,\xxx)$) node[anchor=north] {$y$};
\draw[->,dotted] (0,0) -- ($2.5*(\yy,\yyy)$) node[anchor=north] {$x$};
\draw[->,dotted] (0,0) -- ($2.75*(\zz,\zzz)$) node[anchor=west] {$z$};
%x+y-z=1
\coordinate (a4) at ($0*(\xx,\xxx)+1*(\yy,\yyy)+0*(\zz,\zzz)$);
\coordinate (b4) at ($1*(\xx,\xxx)+1*(\yy,\yyy)+1*(\zz,\zzz)$);
\coordinate (c4) at ($1*(\xx,\xxx)+0*(\yy,\yyy)+0*(\zz,\zzz)$);
\coordinate (d4) at ($1*(\xx,\xxx)+2*(\yy,\yyy)+2*(\zz,\zzz)$);
\coordinate (e4) at ($0*(\xx,\xxx)+2*(\yy,\yyy)+1*(\zz,\zzz)$);
\node[fill=red,draw,circle,inner sep=.5ex] at (a4) {};
\node[fill=red,draw,circle,inner sep=.5ex] at (b4) {};
\node[fill=red,draw,circle,inner sep=.5ex] at (c4) {};
\node[fill=red,draw,circle,inner sep=.5ex] at (d4) {};
\node[fill=red,draw,circle,inner sep=.5ex] at (e4) {};
\draw[fill=light-gray,opacity=.4] (a4) -- (c4) -- (b4) -- (d4) -- (e4) -- cycle;
\node at (0,-1) {$x+y-z=1$};
\end{scope}

\begin{scope}[xshift=5cm,yshift=0cm]
\planepartitionQ{{2,2}}
\draw[->,dotted] (0,0) -- ($1.5*(\xx,\xxx)$) node[anchor=north] {$y$};
\draw[->,dotted] (0,0) -- ($2.5*(\yy,\yyy)$) node[anchor=north] {$x$};
\draw[->,dotted] (0,0) -- ($2.75*(\zz,\zzz)$) node[anchor=west] {$z$};
%x+y-z=2
\coordinate (a5) at ($0*(\xx,\xxx)+2*(\yy,\yyy)+0*(\zz,\zzz)$);
\coordinate (b5) at ($1*(\xx,\xxx)+1*(\yy,\yyy)+0*(\zz,\zzz)$);
\coordinate (c5) at ($1*(\xx,\xxx)+2*(\yy,\yyy)+1*(\zz,\zzz)$);
\node[fill=red,draw,circle,inner sep=.5ex] at (a5) {};
\node[fill=red,draw,circle,inner sep=.5ex] at (b5) {};
\node[fill=red,draw,circle,inner sep=.5ex] at (c5) {};
\draw[fill=light-gray,opacity=.4] (a5) -- (c5) -- (b5) -- cycle;
\node at (0,-1) {$x+y-z=2$};

\end{scope}

\begin{scope}[xshift=0cm,yshift=0cm]
\planepartitionQ{{2,2}}
\draw[->,dotted] (0,0) -- ($1.5*(\xx,\xxx)$) node[anchor=north] {$y$};
\draw[->,dotted] (0,0) -- ($2.5*(\yy,\yyy)$) node[anchor=north] {$x$};
\draw[->,dotted] (0,0) -- ($2.75*(\zz,\zzz)$) node[anchor=west] {$z$};
%x+y-z=3
\coordinate (a6) at ($1*(\xx,\xxx)+2*(\yy,\yyy)+0*(\zz,\zzz)$);
\node[fill=red,draw,circle,inner sep=.5ex] at (a6) {};
\node at (0,-1) {$x+y-z=3$};
\end{scope}
\end{tikzpicture}
\caption{The affine hyperplane toggles %$T_{id,v}^3T_{id,v}^2T_{id,v}^1T_{id,v}^0T_{id,v}^{-1}T_{id,v}^{-2}$ 
corresponding to $\pro_{{\rm id},(1,1,-1)}$ for the identity three-dimensional lattice projection of the poset $J(\mathbf{3}\times\mathbf{2}\times\mathbf{3})$}
\label{fig:hyperplanes}
\end{figure}
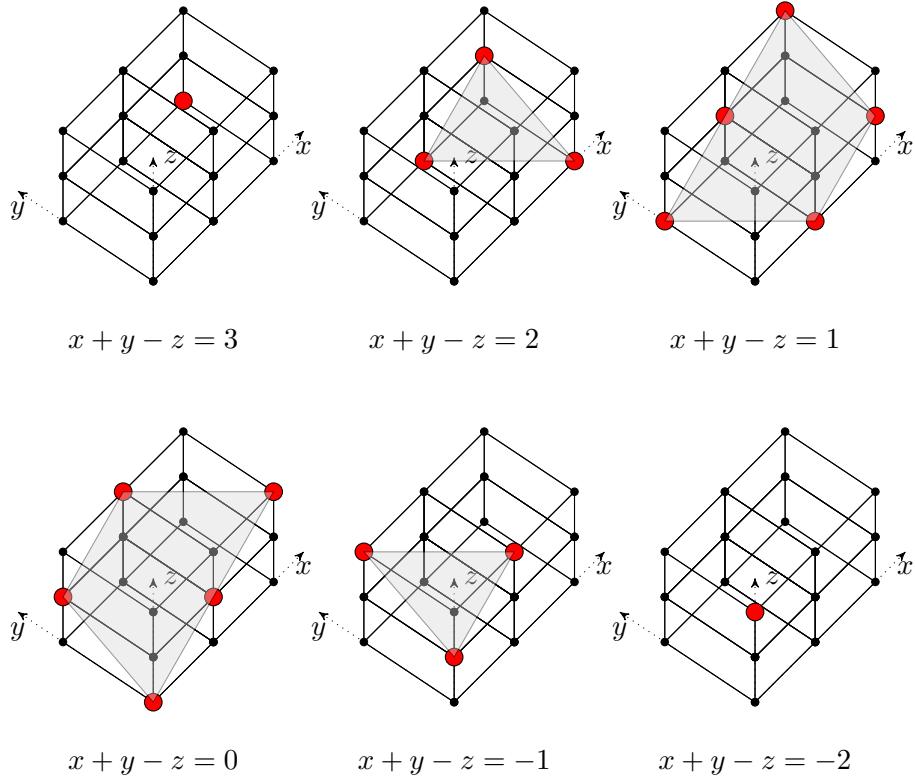

For ease of notation, we may suppress explicitly listing the lattice projection map $\pi$ or the direction $v$ when referring to the generalized promotion operator, if it is clear from context.
%\begin{remark}
Note that for a finite poset $P$,
%has finite support, that is, $\mathrm{supp}(P,\pi,v)$ will necessarily be finite for any choice of $\pi$, $v$.
$T_{\pi,v}^i$ will be the identity operator for all but finitely many $i$.
%\end{remark}

%We compare the notation of Definition~\ref{def:3DPro} with that of~\cite{prorow} in the following remark.

\begin{remark}
To compare with the notion of promotion and rowmotion given in~\cite{prorow}, for a given 2-dimensional lattice projection $\pi$ of a finite poset $P$, rowmotion corresponds to $\pro_{\pi,(1,1)}$, and promotion corresponds to $\pro_{\pi,(1,-1)}$.
\end{remark}

%For an $n$-dimensional lattice projection, rowmotion is a distinguished toggle group action.
\begin{proposition}
For any finite ranked poset $P$ and lattice projection $\pi$, $\pro_{\pi,(1,1,\ldots,1)} = \row$.
\end{proposition}

\begin{proof}
$\pro_{\pi,(1,1,\ldots,1)}$ sweeps through $P$ from top to bottom (in the reverse order of a linear extension), so by Theorem~\ref{thm:row_tog}, this is rowmotion.
\end{proof}

We give some further definitions and lemmas, in order to state and prove Theorem~\ref{thm:3Dprorow} in full generality.

\begin{definition}
Let $P$ be a poset, and let $\pi$, $v$, and $T_{\pi,v}^i$ be as in Definition~\ref{def:3DPro}. Define the \textbf{support} of $(P,\pi,v)$, denoted  $\mathrm{supp}(P,\pi,v)$, to be the smallest interval $[a,b]\subseteq\mathbb{Z}$ such that $T_{\pi,v}^i$ is the identity operator for all $i\in\mathbb{Z}\setminus[a,b]$. 
\end{definition}

\begin{definition}
If $(P,\pi,v)$ has finite support, that is,
$\mathrm{supp}(P,\pi,v)=[a,b]\subset\mathbb{Z}$, let $\sigma:[a,b]\rightarrow[a,b]$ be a bijection. Then define \textbf{promotion with respect to $P$, $\pi$, $v$, and $\sigma$} as the following product of hyperplane-toggles:
\begin{center}
$\pro_{\pi,v}^{\sigma}=T_{\pi,v}^{\sigma(a)}T_{\pi,v}^{\sigma(a+1)}\ldots T_{\pi,v}^{\sigma(b-1)}T_{\pi,v}^{\sigma(b)}$.
\end{center}
\end{definition}

We will use the following toggle group element in the proof of Theorem~\ref{thm:3Dprorow}.
\begin{definition}
\label{def:gyr}
For a poset $P$, define the \textbf{parity} of $p\in P$ as even (resp. odd) if the parity of $\rk(p)$ is even (resp. odd).
Define \textbf{gyration} $\mathrm{Gyr}$ as the toggle group element which first toggles all $p\in P$ with even parity, then all $p$ with odd parity.
\end{definition}

\begin{remark}
Given a lattice projection $\pi$, the rank of $p$ is the same as the rank of $\pi(p)=(x_1,x_2,\ldots x_n)$, which is $\sum_{i}x_i$. Since all the coordinates in $v$ are $\pm 1$, the parity of $\pi(p)$ will be the same as the parity of $\langle \pi(p),v\rangle$. Thus, all elements lying on the same affine hyperplane with respect to $v$ will have the same parity.
\end{remark}

\begin{lemma}
\label{lem:gyr}
If $(P,\pi,v)$ has finite support $[a,b]$, then for any bijection $\sigma:[a,b]\rightarrow[a,b]$ such that $\sigma(k)$ is odd if $k<\frac{a+b}{2}$ and even if $k>\frac{a+b}{2}$, we have $\pro_{\pi,v}^{\sigma}=\mathrm{Gyr}$.
\end{lemma}

We are now nearly ready to state and prove the main theorem of this section. 
We will need the following lemma, which appears as \cite[Lemma~5.1]{prorow}. 
\begin{lemma}[\cite{hoffman.humphreys}]
\label{lem:conjugatelem}
Let $G$ be a group whose generators $g_1,\ldots, g_n$ satisfy $g_i^2 = 1$ and $(g_i g_j)^2 = 1$ if
$|i - j| > 1$. Then for any $\sigma, \tau \in \mathfrak{S}_n$,
$\prod_i g_{\sigma(i)}$ and $\prod_i g_{\tau(i)}$are conjugate.
\end{lemma}

The main theorem of this section is below, whose proof follows the proof of \cite[Theorem~5.2]{prorow}.
\begin{theorem}
\label{thm:3Dprorow}
Let $P$ be a finite poset with an $n$-dimensional lattice projection $\pi$. Let  $v=(v_1,v_2,v_3,\ldots, v_{n})$ and $w=(w_1,w_2,w_3,\ldots, w_{n})$, where $v_j, w_j\in\{\pm 1\}$. Finally suppose that $\sigma:\mathrm{supp}(P,\pi,v)\rightarrow\mathrm{supp}(P,\pi,v)$ and $\tau:\mathrm{supp}(P,\pi,w)\rightarrow\mathrm{supp}(P,\pi,w)$ are bijections.
Then there is an equivariant bijection
between $J(P)$ under $\pro_{\pi,v}^{\sigma}$ and $J(P)$ under $\pro_{\pi,w}^{\tau}$.
\end{theorem}

\begin{proof}
Suppose $P$ is a finite poset with an $n$-dimensional lattice projection $\pi$. Let  $v=(v_1,v_2,v_3,\ldots, v_{n})$,
%and $w=(1,w_2,w_3,\ldots, w_{n})$, 
 where $v_j \in\{\pm 1\}$.
We claim the toggles $T_{\pi,v}^i$ for $i\in\mathrm{supp}(P,\pi,v)$ satisfy the conditions of Lemma~\ref{lem:conjugatelem}. By Lemma~\ref{lem:hyperplane}, $(T_{\pi,v}^i)^2=1$. Also, if $\langle \pi(x),v\rangle=i$ and $\langle \pi(y),v\rangle=j$, then $\langle \pi(y)-\pi(x),v\rangle=j-i$. So if $|j-i|>1$, as all the coefficients in $v$ are $\pm 1$, then $\pi(y)-\pi(x)$ cannot be $e_i$ for any $i$, and $y$ and $x$ cannot be part of a covering relation. Thus, toggles on non-adjacent hyperplanes commute, and we have $(T_{\pi,v}^i T_{\pi,v}^j)^2=1$ when $|j-i|>1$.
So by Lemma~\ref{lem:conjugatelem}, for any bijections $\sigma,\sigma':\mathrm{supp}(P,\pi,v)\rightarrow\mathrm{supp}(P,\pi,v)$,
there is an equivariant bijection between
$J(P)$ under $\pro_{\pi,v}^{\sigma}$ and $J(P)$ under $\pro_{\pi,v}^{\sigma'}$ (since such bijections can be considered as permutations in $\mathfrak{S}_{b-a+1}$ if $\mathrm{supp}(P,\pi,v)=[a,b]$).

Consider $\mathrm{Gyr}$ of Definition~\ref{def:gyr}.
By Lemma~\ref{lem:gyr}, for every $v$ there exists a $\sigma_v$ such that $\mathrm{Gyr}$ can be realized as $\pro_{\pi,v}^{\sigma_v}$. Therefore, 
%for fixed finite poset $P$ and lattice projection $\pi$ and a given $v$, $\sigma$,
there is an equivariant bijection between $J(P)$ under $\pro_{\pi,v}^{\sigma}$ and under $\pro_{\pi,v}^{\sigma_v}=\mathrm{Gyr}$, from which the theorem follows.
%among $J(P)$ under $\pro_{\pi,v}^{\sigma}$ for any choice of $v$ and $\sigma$.
\end{proof}

After we see a bijection between increasing tableaux and plane partitions given in the next section, we will use Theorem~\ref{thm:3Dprorow} to give an improvement on Theorem~\ref{thm:cfdf} of P.~Cameron--D.~Fon-der-Flaass (discussed in Section~\ref{subsec:row}).

\section{An equivariant bijection between plane partitions  and increasing tableaux}
\label{sec:pp_inctab}

\subsection{Bijections between plane partitions and increasing tableaux}
\label{sec:bij}

In this section, we introduce bijections between increasing tableaux and plane partitions. The existence of these bijections should not be at all surprising. However, these maps have amazing properties that will be key to many of our results. These maps are also fundamental to \cite{HPPW}, where they are used to give the first bijective proofs of various results on plane partitions, including R.~Proctor's main result from \cite{Proctor}.

We define a map $\Psi_3 : J({\bf a} \times {\bf b} \times {\bf c}) \to \inc{a \times b}{a+b+c-1}$ as follows.
Let $I \in J({\bf a} \times {\bf b} \times {\bf c})$. Thinking of $I$ in the standard way as a pile of small cubes in an $a \times b \times c$ box, project onto the $a \times b$ face. Record in position $(i,j)$ the number of boxes of $I$ with coordinate $(i,j,k)$ for some $0 \leq k \leq c-1$. The result is a standard plane partition representation of $I$, as a filling of the Young diagram $a \times b$ with nonnegative integers such that rows weakly decrease from left to right and columns weakly decrease from top to bottom. Rotate this plane partition $180^\circ$, so that rows and columns become weakly \emph{increasing}. Now thinking of $a \times b$ as a graded poset with the upper left corner box the unique element of rank $0$, add to each label its rank plus 1. That is, increase each label by one more than its distance from the upper left corner box. (This is  the standard way of converting a \emph{weakly} increasing sequence into a \emph{strictly} increasing one.) The result is the increasing tableau $\Psi_3(I)$. For an example of this transformation, see Figure~\ref{fig:incretization}.

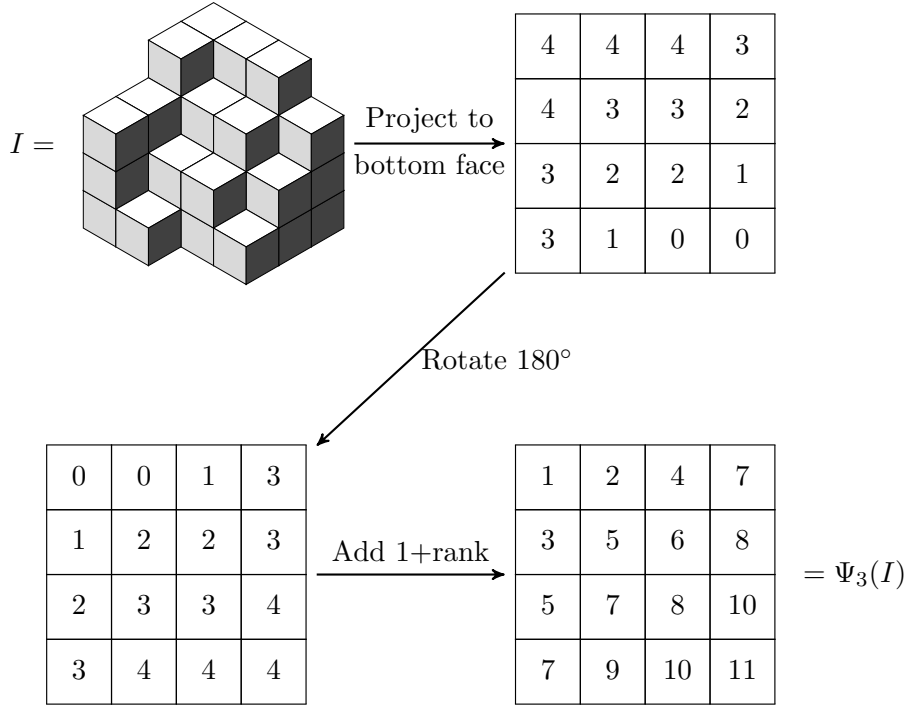
\begin{figure}[ht]
\begin{tikzpicture}
\node (A) {\Scale[.5]{\begin{tikzpicture}
 \planepartition{{4,4,4,3},{4,3,3,2},{3,2,2,1},{3,1}}
\end{tikzpicture}}};
\node[left = .1 of A] (Z) {$I=$};
\node[right = 2 of A] (B) {\ytableaushort{4443,4332,3221,3100}};
\node[below left = 2 and 2.5 of B] (C) {\ytableaushort{0013,1223,2334,3444}};
\node[below = 2 of B] (D) {\ytableaushort{1247,3568,578{10},79{10}{11}}};
\node[right=.1 of D] {$= \Psi_3(I)$};
\path (A) edge[pil] node[above]{Project to} node[below]{bottom face} (B);
\path (B) edge[pil] node[right]{Rotate $180^\circ$} (C);
\path (C) edge[pil] node[above]{Add $1+$rank} (D);
\end{tikzpicture}
\caption{The process of applying $\Psi_3$ to the illustrated $I \in J({\bf 4} \times {\bf 4} \times {\bf 4})$. Here we think of $\Psi_3$ as projecting onto the bottom face of the large bounding box.}
\label{fig:incretization}
\end{figure}

\begin{theorem}
\label{thm:mainbij}
$\Psi_3 : J({\bf a} \times {\bf b} \times {\bf c}) \to \inc{a \times b}{a+b+c-1}$ gives a bijection between plane partitions inside an $a\times b\times c$ box and increasing tableaux of shape $a\times b$ and entries at most $a+b+c-1$.
\end{theorem}
\begin{proof}
The map is defined as the composition of a projection, a rotation, and entrywise addition, all of which are clearly invertible.
\end{proof}

Similarly, define bijections $\Psi_2 : J({\bf a} \times {\bf b} \times {\bf c}) \to \inc{a \times c}{a+b+c-1}$ and $\Psi_1 : J({\bf a} \times {\bf b} \times {\bf c}) \to \inc{b \times c}{a+b+c-1}$ projecting onto the ${\bf a} \times {\bf c}$ and ${\bf b} \times {\bf c}$ faces, respectively (cf.\ Figure~\ref{fig:3projections}).

\begin{figure}[htbp]
\begin{tikzpicture}[scale=1,every node/.append style={transform shape}]
\begin{scope}[scale=.7]
\node (A) {};
 \planepartitionlabel{{4,4,4,3},{4,3,3,2},{3,2,2,1},{3,1}}
 \leftsidelabel{-1}{3}{3}
 \topsidelabel{3}{3}{-1}
 \topsidelabel{3}{2}{-1}
 \rightsidelabel{3}{-1}{3}
 \rightsidelabel{2}{-1}{3}
 \end{scope}
\begin{scope}[xshift=5cm,yshift=-6cm,scale=.8]
\node (L) {\ytableaushort[\color{blue}]{4442,4321,3321,3210}};
\end{scope}
\begin{scope}[yshift=-6cm,scale=.8]
\node (T) {\ytableaushort[\color{red}]{4443,4332,3221,3100}};
\end{scope}
\begin{scope}[xshift=-5cm,yshift=-6cm,scale=.8]
\node (R){\ytableaushort[\color{green}]{4443,4331,4310,2100}};
\end{scope}
\begin{scope}[xshift=5cm,yshift=-10cm,scale=.8]
\node (LL) {\ytableaushort[\color{blue}]{1357,3578,468{10},69{10}{11}}};
\end{scope}
\begin{scope}[yshift=-10cm,scale=.8]
\node (TT) {\ytableaushort[\color{red}]{1247,3568,578{10},79{10}{11}}};
\end{scope}
\begin{scope}[xshift=-5cm,yshift=-10cm,scale=.8]
\node (RR){\ytableaushort[\color{green}]{1246,2479,478{10},79{10}{11}}};
\end{scope}
\draw[->,shorten <=3cm] (A) -- (T);
\draw[->,shorten <=3cm] (A) -- (L);
\draw[->,shorten <=3cm] (A) -- (R);
\draw[->] (T) -- (TT);
\draw[->] (L) -- (LL);
\draw[->] (R) -- (RR);
\end{tikzpicture}
\caption{The three bijections, $\Psi_1$, $\Psi_3$, and $\Psi_2$}
\label{fig:3projections}
\end{figure}
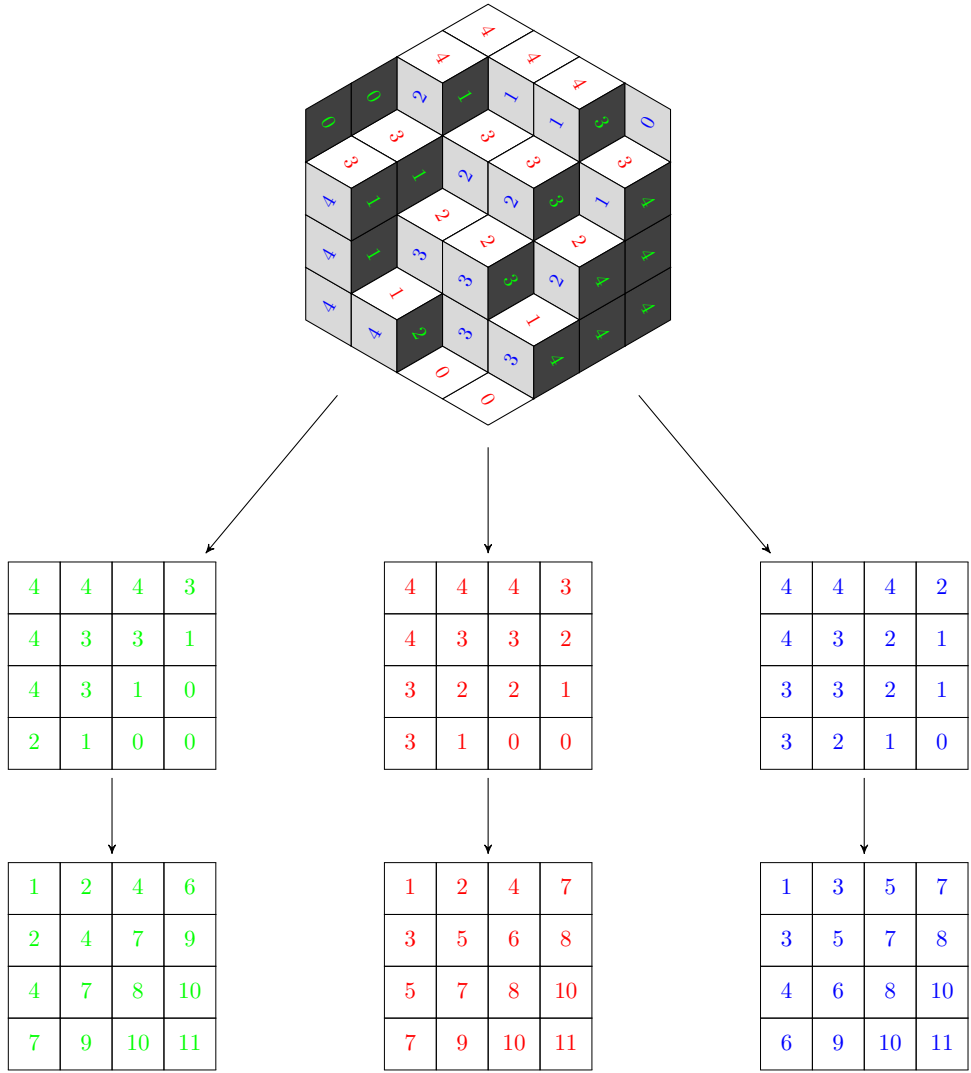

Given the simplicity of the bijection of Theorem~\ref{thm:mainbij}, one might wonder why it was previously overlooked. The set of increasing tableaux in bijection with plane partitions includes those with gaps in the binary content. However much previous research on increasing tableaux was motivated by $K$-theoretic geometry, and in this context there is little reason to consider increasing tableaux without full binary content. Moreover, by restricting to tableaux of full binary content, one obtains some attractive enumerations \cite{pechenik, pressey.stokke.visentin}; for instance, the number of increasing tableaux with shape $2 \times n$ and full binary content is the $n$th \emph{small Schr\"oder number} \cite[Theorem~1.1]{pechenik}.
It was the equivariance of the actions of $\kpro$ and $\row$, discussed in the next section, which led us to observe the bijection of Theorem~\ref{thm:mainbij}.

\subsection{The equivariance of $\kpro$ and $\row$}
\label{sec:equivariant_bijection}

Our first main result was Theorem~\ref{thm:3Dprorow}, that given a poset $P$ with lattice projection $\pi$, there is an equivariant bijection between the order ideals $J(P)$ under $\pro_{\pi,v}^{\sigma}$ and $\pro_{\pi,w}^{\tau}$, where $\sigma, \tau$ are any permutations  of the hyperplane toggles associated to the $\{-1,1\}$-vectors $v, w$.
In this section, we use Theorem~\ref{thm:3Dprorow} in our proof of our second main result, 
%Theorem~\ref{thm:main}, which says that the bijection of Theorem~\ref{thm:mainbij} is equivariant. Then in Corollary
Theorem~\ref{cor:rowinc}, 
%we combine these two theorems to show 
that $\kpro$ and $\row$ are in equivariant bijection.

\begin{lemma}\label{thm:main}
$\Psi_3$ intertwines $\pro_{\mathrm{id},(1,1,-1)}$ and $\kpro$. That is, the following diagram commutes:

\begin{center}
\begin{tikzpicture}
\node (A) {$J({\bf a} \times {\bf b} \times {\bf c})$};
\node[right=2 of A] (B) {$\inc{a \times b}{a+b+c-1}$};
\node[below =2 of A] (C) {$J({\bf a} \times {\bf b} \times {\bf c})$};
\node[below right = 2 and 2 of A] (D) {$\inc{a \times b}{a+b+c-1}$};
\path (A) edge[pil] node[below]{$\Psi_3$} (B);
\path (A) edge[pil] node[right]{$\pro_{\mathrm{id},(1,1,-1)}$} (C);
\path (B) edge[pil] node[right]{$\kpro$} (D);
\path (C) edge[pil] node[below]{$\Psi_3$} (D);
\end{tikzpicture}
\end{center}
\end{lemma}
%%\kevin{We need to make sure the above diagram is  in the right spot}

\begin{proof}
Let $I \in J({\bf a} \times {\bf b} \times {\bf c})$ and let $T = \Psi_3(I)$. Note that the poset ${\bf a} \times {\bf b} \times {\bf c}$ has a $3$-dimensional lattice projection, in the sense of Definition~\ref{def:lattice_proj}, given by the identity map.

By Proposition~\ref{prop:kbk}, $\kpro(T) = \kbk_{a+b+c-2} \circ \dots \circ \kbk_1 (T)$.
Similarly, $\pro_{\mathrm{id},(1,1,-1)} = T_{\mathrm{id},(1,1,-1)}^{(a-1)+(b-1)-(a+b+c-2)} \circ \dots \circ T_{\mathrm{id},(1,1,-1)}^{(a-1)+(b-1)- 1}$.

Thus, it suffices to show that 
\[\Psi_3 \left( T_{\mathrm{id},(1,1,-1)}^{(a-1)+(b-1)-\ell}(I) \right) = \kbk_{\ell}(T).\]

By Definition~\ref{def:3DPro}, $T_{\mathrm{id},(1,1,-1)}^{(a-1)+(b-1)-\ell}$ is the product of the toggles $t_x$ for all $x \in {\bf a} \times {\bf b} \times {\bf c}$ lying on the affine hyperplane determined by $\langle x,(1,1,-1)\rangle=(a-1)+(b-1)-\ell$. Consider $x=(i,j,k)$ on this hyperplane. Then $i+j-k=(a-1)+(b-1)-\ell$. 

We have $x=(i,j,k) \in I$ if and only if the $(a-i,b-j)$ entry of $T$ is at least $k+(a-i)+(b-j)-1=k+a+b-i-j-1$. Since $k=i+j-(a-1)-(b-1)+\ell$, we can rewrite this condition as the $(a-i,b-j)$ entry of $T$ being at least $(i+j-(a-1)-(b-1)+\ell)+a+b-i-j-1=\ell+1$. Hence $x \in I$ if and only if the $(a-i,b-j)$ entry of $T$ is at least $\ell +1$.

\smallskip
\noindent
{\sf (Case 1: $x \in I$):} If $(i,j,k+1) \in I$, then $x$ is unaffected by the toggle and the $(a-i,b-j)$ entry of $T$ is at least $\ell+2$ and so unaffected by $\kbk_\ell$.

Otherwise $(i,j,k+1) \notin I$ and the $(a-i,b-j)$ entry of $T$ equals $\ell+1$. $\kbk_\ell$ will turn this $\ell+1$ into $\ell$ exactly when neither the $(a-i-1,b-j)$ nor the $(a-i,b-j-1)$ entry of $T$ equals $\ell$.  By increasingness of $T$, neither entry is greater than $\ell$. The $(a-i-1,b-j)$ entry of $T$ is at least $\ell$ exactly when $(i+1,j,k) \in I$. Similarly  the $(a-i,b-j-1)$ entry of $T$ is at least $\ell$ exactly when $(i,j+1,k) \in I$. Hence $\kbk_\ell$ will turn this $\ell+1$ into $\ell$ exactly when neither $(i+1,j,k)$ nor $(i,j+1,k)$ is in $I$. But this is exactly when the hyperplane toggle removes $x$ from $I$. Since $I$ is an order ideal, $(i,j,k-1) \in I$, so if $T_{\mathrm{id},(1,1,-1)}^{(a-1)+(b-1)-\ell}$ removes $x$ from $I$, then the $(a-i,b-j)$ entry of $\Psi_3 \left(T_{\mathrm{id},(1,1,-1)}^{(a-1)+(b-1)-\ell}(I) \right)$ equals $\ell$, as desired.

\smallskip
\noindent
{\sf (Case 2: $x \notin I$):}
The $(a-i,b-j)$ entry of $T$ is at most $\ell$. If it is less than $\ell$, then $(i,j,k-1) \notin I$. Hence $x$ is unaffected by the hyperplane toggle and the $(a-i,b-j)$ entry of $T$ is unaffected by $\kbk_\ell$.

Otherwise, the $(a-i,b-j)$ entry of $T$ equals $\ell$ and $(i,j,k-1) \in I$. $\kbk_\ell$ will turn this $\ell$ into $\ell+1$ exactly when neither the $(a-i+1,b-j)$ nor the $(a-i,b-j+1)$ entry of $T$ equals $\ell+1$. This happens exactly when both $(i-1,j,k) \in I$ and $(i,j-1,k) \in I$. Thus $T_{\mathrm{id},(1,1,-1)}^{(a-1)+(b-1)-\ell}$ toggles $x$ into $I$ exactly when $\kbk_\ell$ turns this $\ell$ into $\ell+1$.
\end{proof}

\begin{remark}
\label{r:symm}
By symmetry of $J({\bf a} \times {\bf b} \times {\bf c})$, we obtain analogous results for $\Psi_1$ and $\Psi_2$.
\end{remark}

As a consequence of the above lemma and Theorem~\ref{thm:3Dprorow}, we obtain the following.
\begin{theorem}
\label{cor:rowinc}
$J({\bf a} \times {\bf b} \times {\bf c})$ under $\row$ is in equivariant bijection with $\inc{a \times b}{a+b+c-1}$ under $\kpro$.
\end{theorem}

\subsection{Consequences of the bijection}
\label{subsec:conseq}
In this subsection, we give a number of consequences of Theorem~\ref{cor:rowinc}.
We first give another statement of resonance on plane partitions in Corollary~\ref{cor:row_resonance} (cf.\ Theorem~\ref{cor:resonance_of_pp_directly_row}). In Corollary~\ref{cor:trifold_sym}, we give $\kpro$-equivariant bijections between various sets of increasing tableaux using the tri-fold symmetry of $J({\bf a} \times {\bf b} \times {\bf c})$. We exploit this symmetry to prove Corollaries~\ref{cor:kproab} and~\ref{cor:ab1kpro}. We make a conjecture about the order of $\row$ on $J({\bf a} \times {\bf b} \times {\bf 3})$. Finally, in Theorem~\ref{thm:our_cfdf}, we improve the bound of Theorem~\ref{thm:cfdf} of P.~Cameron--D.~Fon-der-Flaass.

We obtain the following statement of resonance of rowmotion on plane partitions as a consequence of  Theorems~\ref{cor:rowinc} 
and~\ref{cor:resonance_of_increasing_tableaux}. Let $d$ be the toggle group element conjugating $\row$ to $\pro_{\mathrm{id},(1,1,-1)}$. (Theorem~\ref{thm:3Dprorow} guarantees the existence of such an element.)
\begin{corollary}
\label{cor:row_resonance}
$(J({\bf a} \times {\bf b} \times {\bf c}), \langle\row\rangle, {\rm Con}\circ \Psi_3 \circ d)$ exhibits resonance with frequency $a+b+c-1$.
\end{corollary}

\begin{remark}
Given the similarity between Corollary~\ref{cor:row_resonance} and %Proposition~\ref{cor:resonance_of_pp_directly}, 
Theorem~\ref{cor:resonance_of_pp_directly_row},
one may ask what the relation between $X_{\max}$ and ${\rm Con}$ may be. %While it is not true that $X_{\max}$ equals ${\rm Con}\circ \Psi_3$, we instead have 
 The statistics $X_{\max}$ and  ${\rm Con}$ do not correspond exactly (via $\Psi_3$) since the action corresponding to $\kpro$ on plane partitions is $\pro_{\mathrm{id},(1,1,-1)}$ rather than $\pro_{\mathrm{id},(1,-1,1)}$, the action (studied in~\cite{prorow}) that cyclically shifts $X_{\max}$.
 Rather, one can see that $X_{\max}$ is the reverse of ${\rm Con}\circ \Psi_2$.%Corollary~\ref{cor:trifold_sym} below implies the equivalence of Corollary~\ref{cor:row_resonance} and Proposition~\ref{cor:resonance_of_pp_directly}.
\end{remark}

We obtain the following corollary via the tri-fold symmetry of $J({\bf a} \times {\bf b} \times {\bf c})$.
\begin{corollary}
\label{cor:trifold_sym}
There are $\kpro$-equivariant bijections between the sets $\inc{a \times b }{a+b+c-1}$, $\inc{a \times c}{a+b+c-1}$, and $\inc{b \times c}{a+b+c-1}$.
\end{corollary}

%\begin{proof}
%By Lemma~\ref{thm:main} and Remark~\ref{r:symm}, $\Psi_2\circ\Psi_3^{-1}$ is a $\kpro$-equivariant bijection between $\inc{a \times b }{a+b+c-1}$ and $\inc{a \times c}{a+b+c-1}$. Similarly, $\Psi_1\circ\Psi_3^{-1}$ is an equivariant bijection between $\inc{a \times b }{a+b+c-1}$ and $\inc{b \times c}{a+b+c-1}$.
%\end{proof}

\begin{proof}

By Lemma~\ref{thm:main}, Remark~\ref{r:symm}, and Theorem~\ref{thm:3Dprorow}, $\Psi_2\circ d_1\circ \Psi_3^{-1}$ is a $\kpro$-equivariant bijection between $\inc{a \times b }{a+b+c-1}$ and $\inc{a \times c}{a+b+c-1}$, where $d_1$ is the toggle group element taking $\pro_{(1,1,-1)}$ to $\pro_{(1,-1,1)}$.

Similarly, $\Psi_1\circ d_2\circ \Psi_3^{-1}$ is an equivariant bijection between $\inc{a \times b }{a+b+c-1}$ and $\inc{b \times c}{a+b+c-1}$, where $d_2$ is the toggle group element taking $\pro_{(1,1,-1)}$ to $\pro_{(-1,1,1)}$. Theorem~\ref{thm:3Dprorow} guarantees the existence of such $d_1$ and $d_2$.

\end{proof}

Theorem~\ref{cor:rowinc} and Corollary~\ref{cor:trifold_sym} allow us to obtain a number of results for small values of~$c$. %Corollaries~\ref{cor:rowab} and~\ref{cor:rowab2} are 
We obtain new proofs of known results Theorems~\ref{thm:square} and~\ref{thm:height2}, while Corollaries~\ref{cor:kproab} and~\ref{cor:ab1kpro} are new.

We use the following trivial facts about the order of $\kpro$ on particular increasing tableaux.
\begin{fact}\label{fact:1row_pro}
The order of $\kpro$ on $\inc{1 \times a}{q}$ is $q$.
\end{fact}

\begin{fact}\label{fact:minimal_tab}
Let $q > a +b - 1$ and let $M \in \inc{a \times b}{q}$ be the boxwise-minimal increasing tableau of shape $a \times b$, that is the tableau with $1$'s along the first antidiagonal, $2$'s along the second antidiagonal, etc. Then the orbit of $M$ under $\kpro$ has cardinality $q$.
\end{fact}

The following is a new proof of Theorem~\ref{thm:square} of A.~Brouwer--A.~Schrijver \cite{brouwer.schrijver}, which we restate for convenience.
\begin{squarethm}
%\label{cor:rowab}
The order of $\row$ on $J({\bf a} \times {\bf b})$ is $a+b$.
\end{squarethm}
\begin{proof}
The order of $\row$ on $J({\bf a} \times {\bf b})$ is the same as the order of $\row$ on $J({\bf a} \times {\bf b} \times {\bf 1})$. By Corollary~\ref{cor:rowinc}, the order of $\row$ on $J({\bf a} \times {\bf b} \times {\bf 1})$ equals the order of $\kpro$ on $\inc{a \times 1}{a+b}$. By Fact~\ref{fact:1row_pro}, the order of $\kpro$ on $\inc{a \times 1}{a+b}$ is $a+b$.
\end{proof}

The following result is new.
\begin{corollary}
\label{cor:kproab}
The order of $\kpro$ on $\inc{a \times b}{a+b}$ is $a + b$.
\end{corollary}
\begin{proof}
%By Theorem~\ref{thm:main}, 
By the tri-fold symmetry of Corollary~\ref{cor:trifold_sym},
there is a $\kpro$-equivariant bijection between the sets $\inc{a \times b}{a+b}$ and $\inc{1 \times a}{a+b}$. The result is then immediate by Fact~\ref{fact:1row_pro}.
\end{proof} 

We can also use Theorem~\ref{cor:rowinc} and Corollary~\ref{cor:trifold_sym} to show that the Theorem~\ref{thm:height2} of P.~Cameron--D.~Fon-der-Flaass \cite{fonderflaass} is equivalent to a theorem of the second author on increasing tableaux, thus providing 
a new proof of Theorem~\ref{thm:height2}. Alternatively, 
one may use Theorem~\ref{thm:height2} along with Theorem~\ref{cor:rowinc} and Corollary~\ref{cor:trifold_sym} to give a new proof of the second author's result.

\begin{height2thm}
%\label{cor:rowab2}
The order of $\row$ on $J({\bf a} \times {\bf b} \times {\bf 2})$ is $a+b+1$.
\end{height2thm}
\begin{proof}
By Theorem~\ref{cor:rowinc} and Corollary~\ref{cor:trifold_sym}, the order of $\row$ on $J({\bf a} \times {\bf b} \times {\bf 2})$ equals the order of $\kpro$ on $\inc{2 \times a}{a+b+1}$. By \cite[Theorem~1.3]{pechenik}, the latter divides $a+b+1$. (The cited paper only considers those increasing tableaux of full binary content; however its proof extends easily to the case of general binary content.) The theorem then follows by Fact~\ref{fact:minimal_tab}.
\end{proof}

The following result is new.
\begin{corollary}
\label{cor:ab1kpro}
The order of $\kpro$ on $\inc{a \times b}{a+b+1}$ is $a + b + 1$.
\end{corollary}
\begin{proof}
By Corollary~\ref{cor:trifold_sym}, there is a $\kpro$-equivariant bijection between the sets $\inc{a \times b}{a+b+1}$ and $\inc{2 \times a}{a+b+1}$. The result is then immediate from \cite[Theorem~1.3]{pechenik} and Fact~\ref{fact:minimal_tab}.
\end{proof}

Recall that for $c >3$, the order of $\row$ on $J({\bf a} \times {\bf b} \times {\bf c})$ is generally greater than $a+b+c-1$. Nonetheless, we make the following conjecture.

\begin{conjecture}\label{conj:three_row}
The order of $\row$ on $J({\bf a} \times {\bf b} \times {\bf 3})$ is $a+b+2$.
\end{conjecture}

As with the above corollaries, the results of this paper show that Conjecture~\ref{conj:three_row} is equivalent to the order of $K$-promotion being $a+b+2$ on either $\inc{a \times b}{a+b+2}$ or $\inc{3 \times a}{a+b+2}$. We have verified Conjecture~\ref{conj:three_row} for $a \leq 7$ and $b$ arbitrary.

Finally, we improve the bound in Theorem~\ref{thm:cfdf} of P.~Cameron--D.~Fon-der-Flaass~\cite{fonderflaass} by more than a factor of $\frac{2}{3}$. This is evidence toward the conjecture of P.~Cameron--D.~Fon-der-Flaass \cite{fonderflaass} that this condition may be dropped entirely.

\begin{theorem}
\label{thm:our_cfdf}
If $a + b + c -1$ is prime and $c > \displaystyle\frac{2ab-2}{3} - a - b + 2$, then the cardinality of every orbit of $\row$ on $J({\bf a} \times {\bf b} \times {\bf c})$ is a multiple of $a+b+c-1$.
\end{theorem}
\begin{proof}
Let $q = a+b+c-1$. The case $q=2$ is trivial, so assume $q$ is odd.

Consider $I \in J({\bf a} \times {\bf b} \times {\bf c})$ and let $T = \Psi_3(I) \in \inc{a \times b}{q}$. If $T$ does not have full binary content, then by Corollary~\ref{cor:weakest_CFdFlike_statement}, the $K$-promotion orbit of $T$ has cardinality a multiple of $q$. Hence by Theorem~\ref{cor:rowinc}, the rowmotion orbit of $I$ has cardinality a multiple of $q$, as claimed.
Thus, we may assume $T$ has full binary content.

Similarly, by Proposition~\ref{prop:content_descents}, we may assume that every $1 \leq i \leq q$ is both a descent and a transpose descent in $T$. Hence by Lemma~\ref{lem:three_for_two}, for $1 \leq j \leq \frac{q-1}{2}$, the number of $(2j-1)$'s in $T$ plus the number of $2j$'s in $T$ is at least $3$. By the increasingness conditions on $T$, there is exactly $1$ instance of $q$ in $T$. Thus the total number of labels in $T$ is at least $3\frac{q-1}{2} + 1$.

Since $T \in \inc{a \times b}{q}$, this forces $3\frac{a+b+c-2}{2} + 1 \leq ab$. Thus $c \leq \frac{2ab -2}{3} - a - b + 2$, contradicting the assumed bound on $c$.
\end{proof} 

\section{Resonance for other combinatorial objects}
\label{subsec:conj}

In this final section, we first give in Corollary~\ref{cor:fpl} an additional example of resonance of the \emph{gyration} action on %\emph{alternating sign matrices} and 
\emph{fully packed loop configurations}. We then state Problems~\ref{prob:spro} and~\ref{prob:tsscpp} on proving new instances of resonance.

%Another example that motivated the definition of resonance is the action of \emph{gyration} on %\emph{alternating sign matrices} and \emph{fully packed loop configurations}.

%Alternating sign matrices were introduced by D.~Robbins--H.~Rumsey \cite{RobbinsRumsey} as part of their study of the lambda-determinant. With W.~Mills \cite{MRRASM}, they then conjectured an enumeration for ${\rm ASM}_n$, which was proved by D.~Zeilberger~\cite{zeilberger} and G.~Kuperberg~\cite{kuperberg} (cf.~\cite{BressoudBook} for a detailed exposition of this history).
%Alternating sign matrices are known to be in bijection with \emph{fully-packed loop configurations}~\cite{wieland, ProppManyFaces}; see Figures~\ref{fig:fplex} and~\ref{fig:fplex2}.
\begin{definition}
\label{def:fpl}
	Consider an $[n] \times [n]$ grid of dots in $\mathbb{Z}^2$ with edges between dots that are horizontally or vertically adjecent. Beginning with the dot at the upper left corner, draw an edge from that dot up one unit. Then go around the grid, drawing such an external edge at every second dot (counting corner dots twice, since at the corners, external edges could go in either of two directions). A {\bf fully-packed loop configuration (FPL)} of order $n$ is a subgraph of the grid graph
    %set of paths and loops on the $[n] \times [n]$ grid with boundary conditions as 
    described above, such that each of the $n^2$ vertices within the grid has exactly two incident edges. Let ${\rm FPL}_n$ be the set of all order $n$ fully packed loop configurations.
\end{definition}

There is a (non-injective) map from fully-packed loop configurations to their \emph{link patterns}. See Figure~\ref{fig:fplex} for an example.
\begin{definition} 
\label{def:lp}
	Given a fully-packed loop configuration, number the external edges clockwise, starting with the upper left external edge. Each external edge will be connected by a path to another external edge, and these paths will never cross.
	This matching on the external edges is a noncrossing matching on $2n$ points, and is called the {\bf link pattern} of the FPL.
\end{definition}

Consider the following action on fully-packed loop configurations; see Figure~\ref{fig:gyration_ex}.
\begin{definition}
\label{def:fplgyr}
Given an $[n]\times[n]$ grid of dots, color the interiors of the squares in a checkerboard pattern.  Given an FPL of order $n$ drawn on this grid, its \textbf{gyration}, {\rm Gyr}, is computed by first visiting all squares of one color then all squares of the other color,
applying at each visited square the ``local move'' that swaps the edges around a square if the edges are parallel and otherwise leaves them fixed.
\end{definition}

\begin{figure}[htb]
\[\includegraphics[scale = .2]{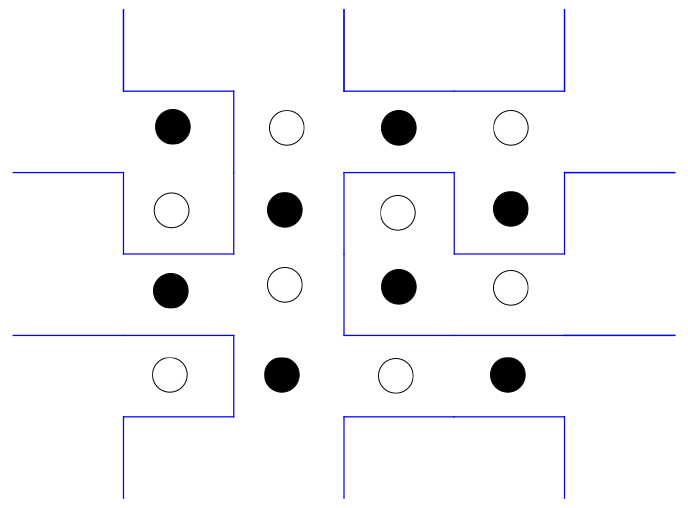}\hspace{2cm} \includegraphics[scale = .2]{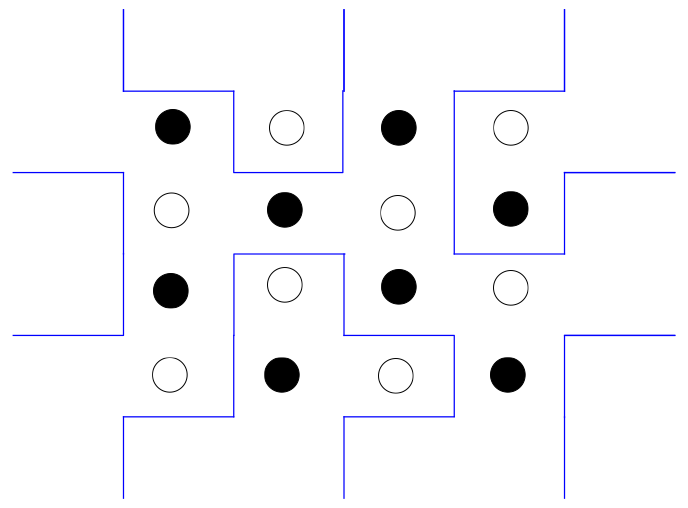}\hspace{1ex}\]
\caption{An example of gyration on the fully-packed loop configuration shown at left. First at each square marked with $\newmoon$, we replace the local configuration $\textcolor{blue}{\overline{\underline{\textcolor{black}{\newmoon}}}}$ with $\textcolor{blue}{|} \newmoon \textcolor{blue}{|}$ and vice versa, obtaining the picture on the right. Then we perform the same local switch at each square marked with $\fullmoon$. In this case, there are no local configurations $\textcolor{blue}{\overline{\underline{\textcolor{black}{\fullmoon}}}}$ or $\textcolor{blue}{|} \fullmoon \textcolor{blue}{|}$ in the picture on the right, so we obtain the fully-packed loop configuration on the right as the result of gyration.}\label{fig:gyration_ex}
\end{figure}

The following theorem of B.~Wieland gives a remarkable property of gyration.
\begin{theorem}[B.~Wieland \cite{wieland}]
Gyration of an FPL rotates the associated link pattern by an angle of $2\pi/2n$.
\end{theorem}

We reformulate this theorem into a statement of resonance.
\begin{corollary}
\label{cor:fpl}
Let $f$ be the map from a fully-packed loop configuration to its link pattern. Then, $({\rm FPL}_n, \langle {\rm Gyr} \rangle, f)$ exhibits resonance with frequency $2n$.
\end{corollary}

For example, consider gyration on $5\times 5$ fully-packed loops. Gyration has orbits of size 2, 4, 5, and 10. So the order of gyration in this case is 20, but $({\rm FPL}_5$, {\rm Gyr}, $f)$ exhibits resonance with frequency $10$. Consider the orbit of gyration in Figure~\ref{fig:fplex}. This orbit is of size 4, while the link pattern orbit is of size 2. So even though ${\rm Gyr}^{10}(A)\neq A$ for $A$ an FPL in this orbit, ${\rm rot}^{10}(f(A))=f(A)$ (since, in this case, ${\rm rot}^{2}(f(A))=f(A)$).

\begin{figure}[htbp]
\begin{center}

\includegraphics[scale = .145]{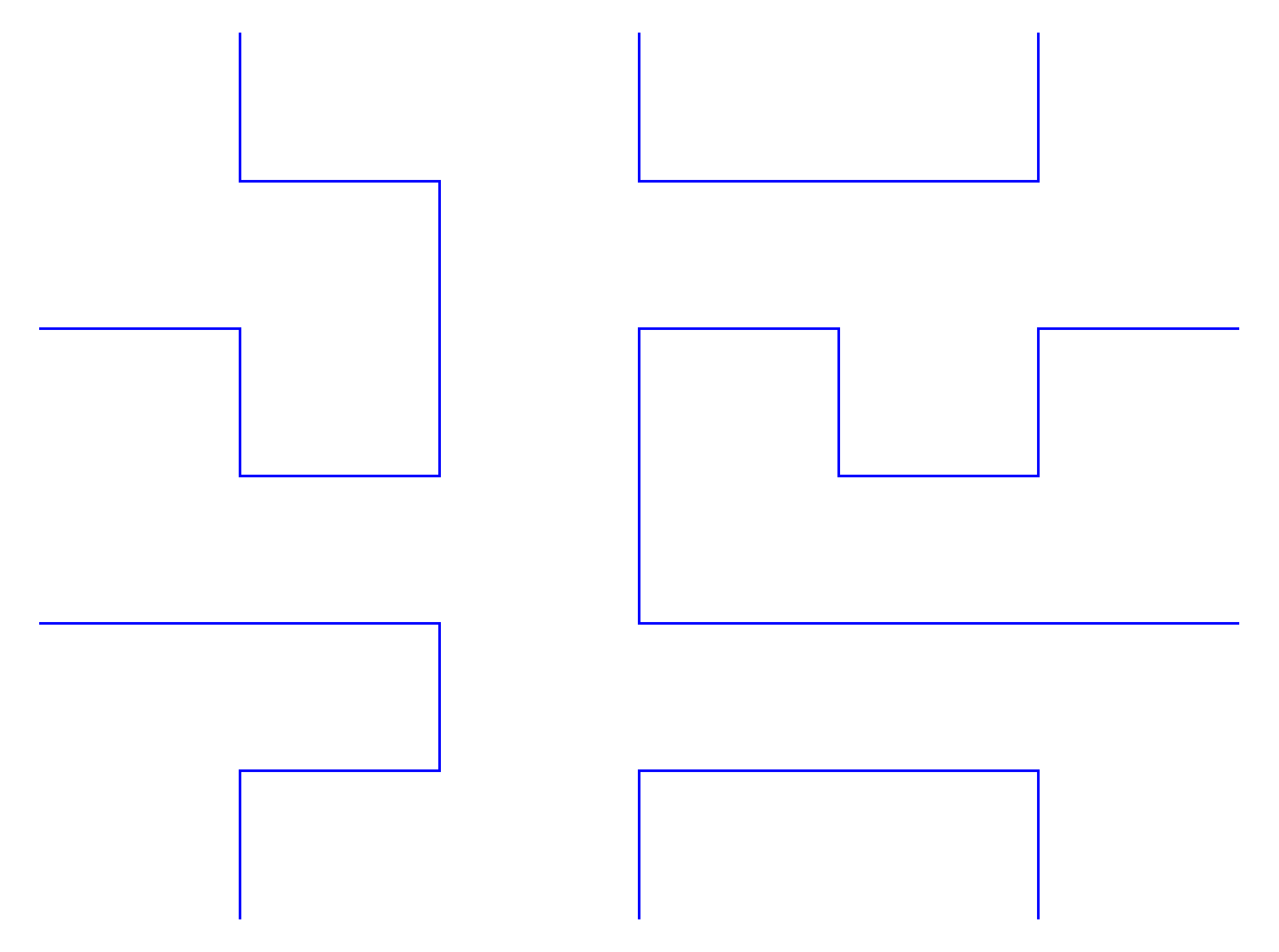}\hspace{1ex}
\includegraphics[scale = .145]{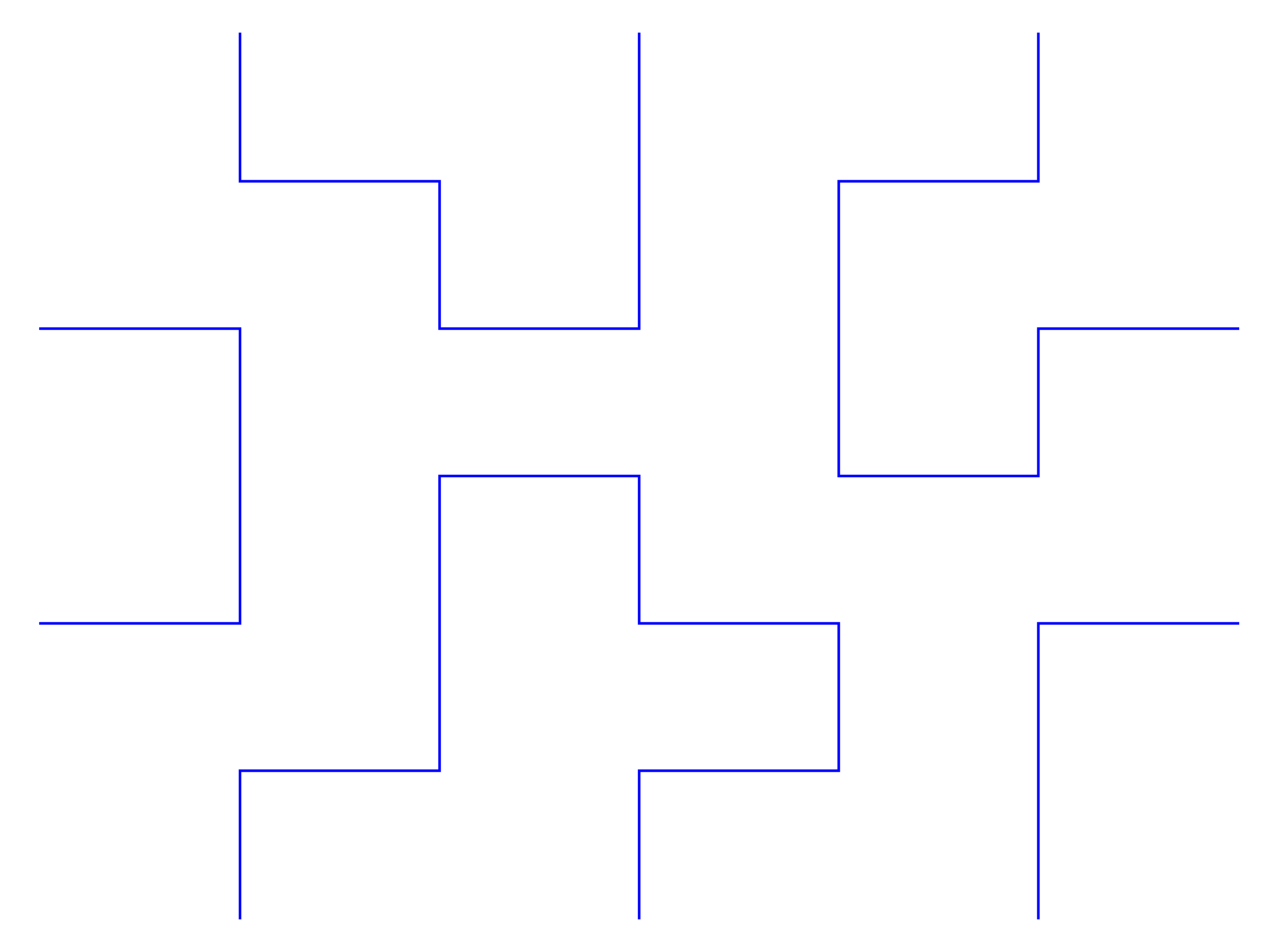}\hspace{1ex}
\includegraphics[scale = .145]{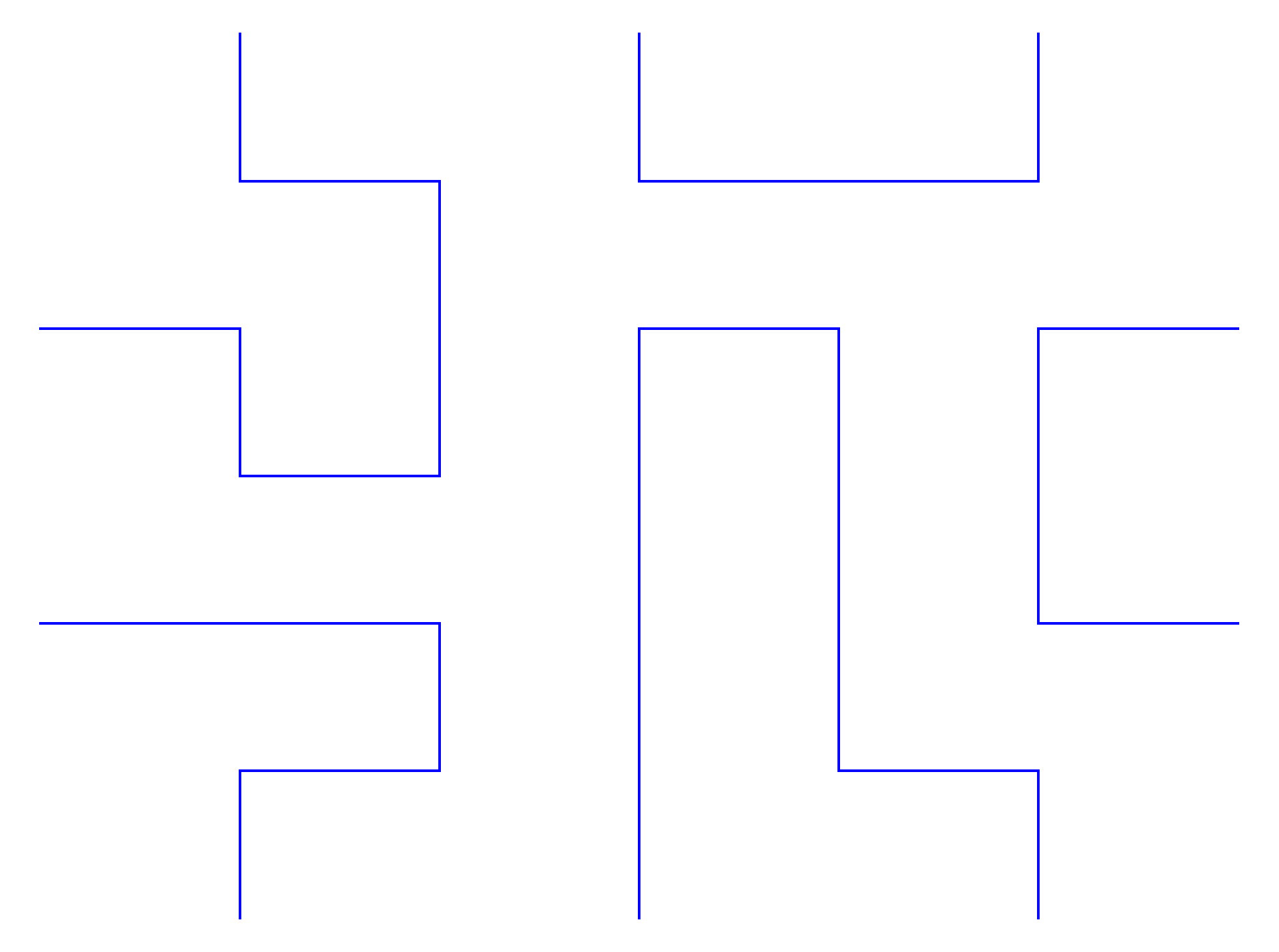}\hspace{1ex}
\includegraphics[scale = .145]{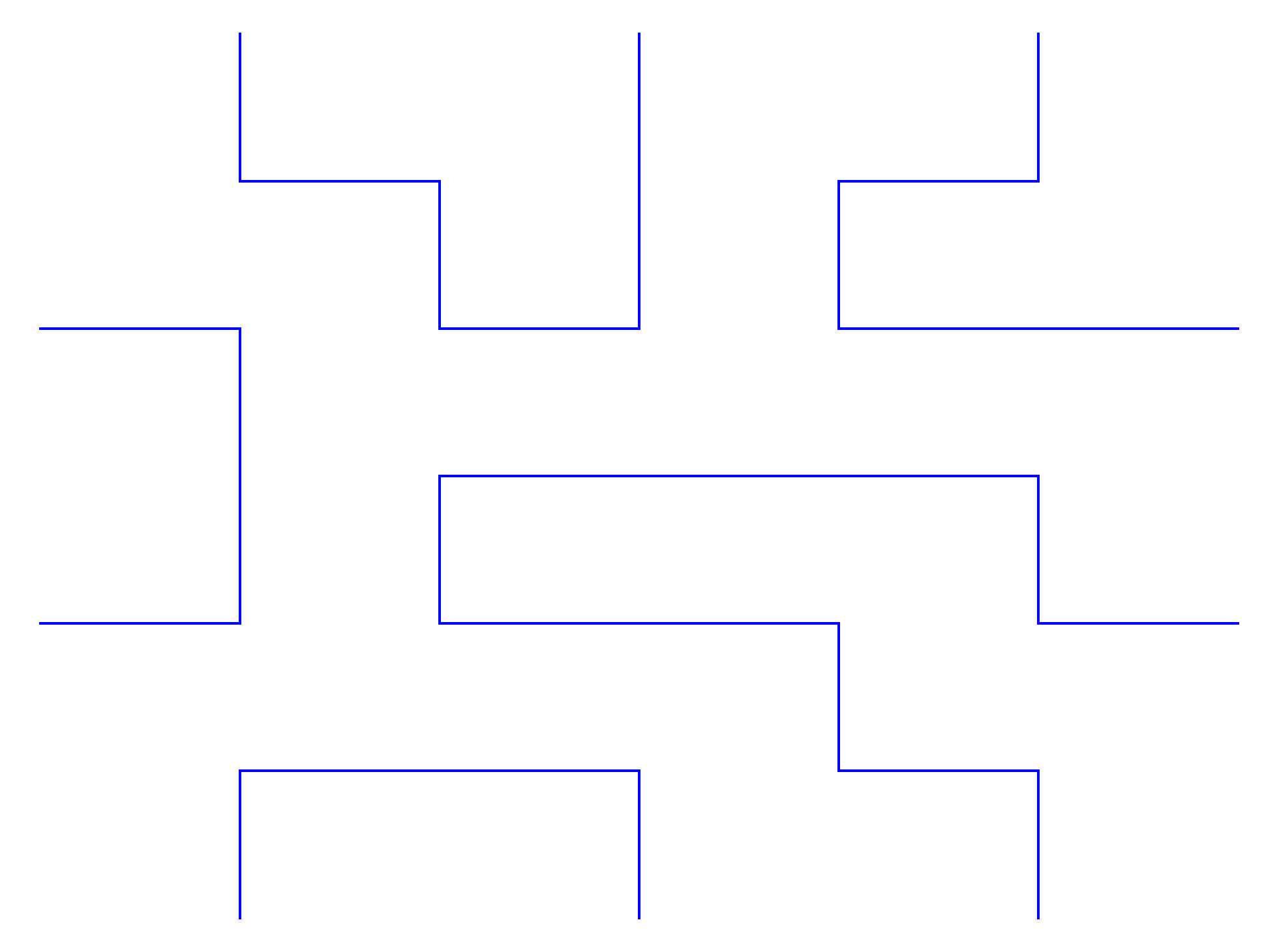}

\vspace{1ex}

\scalebox{.5}{
\begin{tikzpicture}
\useasboundingbox (0,0) rectangle (5.0cm,5.0cm);
\definecolor{cv0}{rgb}{0.0,0.0,0.0}
\definecolor{cfv0}{rgb}{1.0,1.0,1.0}
\definecolor{clv0}{rgb}{0.0,0.0,0.0}
\definecolor{cv1}{rgb}{0.0,0.0,0.0}
\definecolor{cfv1}{rgb}{1.0,1.0,1.0}
\definecolor{clv1}{rgb}{0.0,0.0,0.0}
\definecolor{cv2}{rgb}{0.0,0.0,0.0}
\definecolor{cfv2}{rgb}{1.0,1.0,1.0}
\definecolor{clv2}{rgb}{0.0,0.0,0.0}
\definecolor{cv3}{rgb}{0.0,0.0,0.0}
\definecolor{cfv3}{rgb}{1.0,1.0,1.0}
\definecolor{clv3}{rgb}{0.0,0.0,0.0}
\definecolor{cv4}{rgb}{0.0,0.0,0.0}
\definecolor{cfv4}{rgb}{1.0,1.0,1.0}
\definecolor{clv4}{rgb}{0.0,0.0,0.0}
\definecolor{cv5}{rgb}{0.0,0.0,0.0}
\definecolor{cfv5}{rgb}{1.0,1.0,1.0}
\definecolor{clv5}{rgb}{0.0,0.0,0.0}
\definecolor{cv6}{rgb}{0.0,0.0,0.0}
\definecolor{cfv6}{rgb}{1.0,1.0,1.0}
\definecolor{clv6}{rgb}{0.0,0.0,0.0}
\definecolor{cv7}{rgb}{0.0,0.0,0.0}
\definecolor{cfv7}{rgb}{1.0,1.0,1.0}
\definecolor{clv7}{rgb}{0.0,0.0,0.0}
\definecolor{cv8}{rgb}{0.0,0.0,0.0}
\definecolor{cfv8}{rgb}{1.0,1.0,1.0}
\definecolor{clv8}{rgb}{0.0,0.0,0.0}
\definecolor{cv9}{rgb}{0.0,0.0,0.0}
\definecolor{cfv9}{rgb}{1.0,1.0,1.0}
\definecolor{clv9}{rgb}{0.0,0.0,0.0}
\definecolor{cv0v9}{rgb}{0.0,0.0,0.0}
\definecolor{cv1v2}{rgb}{0.0,0.0,0.0}
\definecolor{cv3v4}{rgb}{0.0,0.0,0.0}
\definecolor{cv5v6}{rgb}{0.0,0.0,0.0}
\definecolor{cv7v8}{rgb}{0.0,0.0,0.0}
\Vertex[style={minimum size=0.4cm,draw=cv0,fill=cfv0,text=clv0,shape=circle},LabelOut=false,L=\hbox{$2$},x=2.5cm,y=5.0cm]{v0}
\Vertex[style={minimum size=0.4cm,draw=cv1,fill=cfv1,text=clv1,shape=circle},LabelOut=false,L=\hbox{$1$},x=0.9549cm,y=4.5225cm]{v1}
\Vertex[style={minimum size=0.4cm,draw=cv2,fill=cfv2,text=clv2,shape=circle},LabelOut=false,L=\hbox{$10$},x=0.0cm,y=3.2725cm]{v2}
\Vertex[style={minimum size=0.4cm,draw=cv3,fill=cfv3,text=clv3,shape=circle},LabelOut=false,L=\hbox{$9$},x=0.0cm,y=1.7275cm]{v3}
\Vertex[style={minimum size=0.4cm,draw=cv4,fill=cfv4,text=clv4,shape=circle},LabelOut=false,L=\hbox{$8$},x=0.9549cm,y=0.4775cm]{v4}
\Vertex[style={minimum size=0.4cm,draw=cv5,fill=cfv5,text=clv5,shape=circle},LabelOut=false,L=\hbox{$7$},x=2.5cm,y=0.0cm]{v5}
\Vertex[style={minimum size=0.4cm,draw=cv6,fill=cfv6,text=clv6,shape=circle},LabelOut=false,L=\hbox{$6$},x=4.0451cm,y=0.4775cm]{v6}
\Vertex[style={minimum size=0.4cm,draw=cv7,fill=cfv7,text=clv7,shape=circle},LabelOut=false,L=\hbox{$5$},x=5.0cm,y=1.7275cm]{v7}
\Vertex[style={minimum size=0.4cm,draw=cv8,fill=cfv8,text=clv8,shape=circle},LabelOut=false,L=\hbox{$4$},x=5.0cm,y=3.2725cm]{v8}
\Vertex[style={minimum size=0.4cm,draw=cv9,fill=cfv9,text=clv9,shape=circle},LabelOut=false,L=\hbox{$3$},x=4.0451cm,y=4.5225cm]{v9}
\Edge[lw=0.04cm,style={color=cv0v9,},](v0)(v9)
\Edge[lw=0.04cm,style={color=cv1v2,},](v1)(v2)
\Edge[lw=0.04cm,style={color=cv3v4,},](v3)(v4)
\Edge[lw=0.04cm,style={color=cv5v6,},](v5)(v6)
\Edge[lw=0.04cm,style={color=cv7v8,},](v7)(v8)
\end{tikzpicture}
}
\hspace{2ex}
\scalebox{.5}{
\begin{tikzpicture}
\useasboundingbox (0,0) rectangle (5.0cm,5.0cm);
\definecolor{cv0}{rgb}{0.0,0.0,0.0}
\definecolor{cfv0}{rgb}{1.0,1.0,1.0}
\definecolor{clv0}{rgb}{0.0,0.0,0.0}
\definecolor{cv1}{rgb}{0.0,0.0,0.0}
\definecolor{cfv1}{rgb}{1.0,1.0,1.0}
\definecolor{clv1}{rgb}{0.0,0.0,0.0}
\definecolor{cv2}{rgb}{0.0,0.0,0.0}
\definecolor{cfv2}{rgb}{1.0,1.0,1.0}
\definecolor{clv2}{rgb}{0.0,0.0,0.0}
\definecolor{cv3}{rgb}{0.0,0.0,0.0}
\definecolor{cfv3}{rgb}{1.0,1.0,1.0}
\definecolor{clv3}{rgb}{0.0,0.0,0.0}
\definecolor{cv4}{rgb}{0.0,0.0,0.0}
\definecolor{cfv4}{rgb}{1.0,1.0,1.0}
\definecolor{clv4}{rgb}{0.0,0.0,0.0}
\definecolor{cv5}{rgb}{0.0,0.0,0.0}
\definecolor{cfv5}{rgb}{1.0,1.0,1.0}
\definecolor{clv5}{rgb}{0.0,0.0,0.0}
\definecolor{cv6}{rgb}{0.0,0.0,0.0}
\definecolor{cfv6}{rgb}{1.0,1.0,1.0}
\definecolor{clv6}{rgb}{0.0,0.0,0.0}
\definecolor{cv7}{rgb}{0.0,0.0,0.0}
\definecolor{cfv7}{rgb}{1.0,1.0,1.0}
\definecolor{clv7}{rgb}{0.0,0.0,0.0}
\definecolor{cv8}{rgb}{0.0,0.0,0.0}
\definecolor{cfv8}{rgb}{1.0,1.0,1.0}
\definecolor{clv8}{rgb}{0.0,0.0,0.0}
\definecolor{cv9}{rgb}{0.0,0.0,0.0}
\definecolor{cfv9}{rgb}{1.0,1.0,1.0}
\definecolor{clv9}{rgb}{0.0,0.0,0.0}
\definecolor{cv0v1}{rgb}{0.0,0.0,0.0}
\definecolor{cv2v3}{rgb}{0.0,0.0,0.0}
\definecolor{cv4v5}{rgb}{0.0,0.0,0.0}
\definecolor{cv6v7}{rgb}{0.0,0.0,0.0}
\definecolor{cv8v9}{rgb}{0.0,0.0,0.0}
\Vertex[style={minimum size=0.4cm,draw=cv0,fill=cfv0,text=clv0,shape=circle},LabelOut=false,L=\hbox{$2$},x=2.5cm,y=5.0cm]{v0}
\Vertex[style={minimum size=0.4cm,draw=cv1,fill=cfv1,text=clv1,shape=circle},LabelOut=false,L=\hbox{$1$},x=0.9549cm,y=4.5225cm]{v1}
\Vertex[style={minimum size=0.4cm,draw=cv2,fill=cfv2,text=clv2,shape=circle},LabelOut=false,L=\hbox{$10$},x=0.0cm,y=3.2725cm]{v2}
\Vertex[style={minimum size=0.4cm,draw=cv3,fill=cfv3,text=clv3,shape=circle},LabelOut=false,L=\hbox{$9$},x=0.0cm,y=1.7275cm]{v3}
\Vertex[style={minimum size=0.4cm,draw=cv4,fill=cfv4,text=clv4,shape=circle},LabelOut=false,L=\hbox{$8$},x=0.9549cm,y=0.4775cm]{v4}
\Vertex[style={minimum size=0.4cm,draw=cv5,fill=cfv5,text=clv5,shape=circle},LabelOut=false,L=\hbox{$7$},x=2.5cm,y=0.0cm]{v5}
\Vertex[style={minimum size=0.4cm,draw=cv6,fill=cfv6,text=clv6,shape=circle},LabelOut=false,L=\hbox{$6$},x=4.0451cm,y=0.4775cm]{v6}
\Vertex[style={minimum size=0.4cm,draw=cv7,fill=cfv7,text=clv7,shape=circle},LabelOut=false,L=\hbox{$5$},x=5.0cm,y=1.7275cm]{v7}
\Vertex[style={minimum size=0.4cm,draw=cv8,fill=cfv8,text=clv8,shape=circle},LabelOut=false,L=\hbox{$4$},x=5.0cm,y=3.2725cm]{v8}
\Vertex[style={minimum size=0.4cm,draw=cv9,fill=cfv9,text=clv9,shape=circle},LabelOut=false,L=\hbox{$3$},x=4.0451cm,y=4.5225cm]{v9}
\Edge[lw=0.04cm,style={color=cv0v1,},](v0)(v1)
\Edge[lw=0.04cm,style={color=cv2v3,},](v2)(v3)
\Edge[lw=0.04cm,style={color=cv4v5,},](v4)(v5)
\Edge[lw=0.04cm,style={color=cv6v7,},](v6)(v7)
\Edge[lw=0.04cm,style={color=cv8v9,},](v8)(v9)
\end{tikzpicture}}
\hspace{2ex}
\scalebox{.5}{
\begin{tikzpicture}
\useasboundingbox (0,0) rectangle (5.0cm,5.0cm);
\definecolor{cv0}{rgb}{0.0,0.0,0.0}
\definecolor{cfv0}{rgb}{1.0,1.0,1.0}
\definecolor{clv0}{rgb}{0.0,0.0,0.0}
\definecolor{cv1}{rgb}{0.0,0.0,0.0}
\definecolor{cfv1}{rgb}{1.0,1.0,1.0}
\definecolor{clv1}{rgb}{0.0,0.0,0.0}
\definecolor{cv2}{rgb}{0.0,0.0,0.0}
\definecolor{cfv2}{rgb}{1.0,1.0,1.0}
\definecolor{clv2}{rgb}{0.0,0.0,0.0}
\definecolor{cv3}{rgb}{0.0,0.0,0.0}
\definecolor{cfv3}{rgb}{1.0,1.0,1.0}
\definecolor{clv3}{rgb}{0.0,0.0,0.0}
\definecolor{cv4}{rgb}{0.0,0.0,0.0}
\definecolor{cfv4}{rgb}{1.0,1.0,1.0}
\definecolor{clv4}{rgb}{0.0,0.0,0.0}
\definecolor{cv5}{rgb}{0.0,0.0,0.0}
\definecolor{cfv5}{rgb}{1.0,1.0,1.0}
\definecolor{clv5}{rgb}{0.0,0.0,0.0}
\definecolor{cv6}{rgb}{0.0,0.0,0.0}
\definecolor{cfv6}{rgb}{1.0,1.0,1.0}
\definecolor{clv6}{rgb}{0.0,0.0,0.0}
\definecolor{cv7}{rgb}{0.0,0.0,0.0}
\definecolor{cfv7}{rgb}{1.0,1.0,1.0}
\definecolor{clv7}{rgb}{0.0,0.0,0.0}
\definecolor{cv8}{rgb}{0.0,0.0,0.0}
\definecolor{cfv8}{rgb}{1.0,1.0,1.0}
\definecolor{clv8}{rgb}{0.0,0.0,0.0}
\definecolor{cv9}{rgb}{0.0,0.0,0.0}
\definecolor{cfv9}{rgb}{1.0,1.0,1.0}
\definecolor{clv9}{rgb}{0.0,0.0,0.0}
\definecolor{cv0v9}{rgb}{0.0,0.0,0.0}
\definecolor{cv1v2}{rgb}{0.0,0.0,0.0}
\definecolor{cv3v4}{rgb}{0.0,0.0,0.0}
\definecolor{cv5v6}{rgb}{0.0,0.0,0.0}
\definecolor{cv7v8}{rgb}{0.0,0.0,0.0}
\Vertex[style={minimum size=0.4cm,draw=cv0,fill=cfv0,text=clv0,shape=circle},LabelOut=false,L=\hbox{$2$},x=2.5cm,y=5.0cm]{v0}
\Vertex[style={minimum size=0.4cm,draw=cv1,fill=cfv1,text=clv1,shape=circle},LabelOut=false,L=\hbox{$1$},x=0.9549cm,y=4.5225cm]{v1}
\Vertex[style={minimum size=0.4cm,draw=cv2,fill=cfv2,text=clv2,shape=circle},LabelOut=false,L=\hbox{$10$},x=0.0cm,y=3.2725cm]{v2}
\Vertex[style={minimum size=0.4cm,draw=cv3,fill=cfv3,text=clv3,shape=circle},LabelOut=false,L=\hbox{$9$},x=0.0cm,y=1.7275cm]{v3}
\Vertex[style={minimum size=0.4cm,draw=cv4,fill=cfv4,text=clv4,shape=circle},LabelOut=false,L=\hbox{$8$},x=0.9549cm,y=0.4775cm]{v4}
\Vertex[style={minimum size=0.4cm,draw=cv5,fill=cfv5,text=clv5,shape=circle},LabelOut=false,L=\hbox{$7$},x=2.5cm,y=0.0cm]{v5}
\Vertex[style={minimum size=0.4cm,draw=cv6,fill=cfv6,text=clv6,shape=circle},LabelOut=false,L=\hbox{$6$},x=4.0451cm,y=0.4775cm]{v6}
\Vertex[style={minimum size=0.4cm,draw=cv7,fill=cfv7,text=clv7,shape=circle},LabelOut=false,L=\hbox{$5$},x=5.0cm,y=1.7275cm]{v7}
\Vertex[style={minimum size=0.4cm,draw=cv8,fill=cfv8,text=clv8,shape=circle},LabelOut=false,L=\hbox{$4$},x=5.0cm,y=3.2725cm]{v8}
\Vertex[style={minimum size=0.4cm,draw=cv9,fill=cfv9,text=clv9,shape=circle},LabelOut=false,L=\hbox{$3$},x=4.0451cm,y=4.5225cm]{v9}
\Edge[lw=0.04cm,style={color=cv0v9,},](v0)(v9)
\Edge[lw=0.04cm,style={color=cv1v2,},](v1)(v2)
\Edge[lw=0.04cm,style={color=cv3v4,},](v3)(v4)
\Edge[lw=0.04cm,style={color=cv5v6,},](v5)(v6)
\Edge[lw=0.04cm,style={color=cv7v8,},](v7)(v8)
\end{tikzpicture}
}
\hspace{2ex}
\scalebox{.5}{
\begin{tikzpicture}
\useasboundingbox (0,0) rectangle (5.0cm,5.0cm);
\definecolor{cv0}{rgb}{0.0,0.0,0.0}
\definecolor{cfv0}{rgb}{1.0,1.0,1.0}
\definecolor{clv0}{rgb}{0.0,0.0,0.0}
\definecolor{cv1}{rgb}{0.0,0.0,0.0}
\definecolor{cfv1}{rgb}{1.0,1.0,1.0}
\definecolor{clv1}{rgb}{0.0,0.0,0.0}
\definecolor{cv2}{rgb}{0.0,0.0,0.0}
\definecolor{cfv2}{rgb}{1.0,1.0,1.0}
\definecolor{clv2}{rgb}{0.0,0.0,0.0}
\definecolor{cv3}{rgb}{0.0,0.0,0.0}
\definecolor{cfv3}{rgb}{1.0,1.0,1.0}
\definecolor{clv3}{rgb}{0.0,0.0,0.0}
\definecolor{cv4}{rgb}{0.0,0.0,0.0}
\definecolor{cfv4}{rgb}{1.0,1.0,1.0}
\definecolor{clv4}{rgb}{0.0,0.0,0.0}
\definecolor{cv5}{rgb}{0.0,0.0,0.0}
\definecolor{cfv5}{rgb}{1.0,1.0,1.0}
\definecolor{clv5}{rgb}{0.0,0.0,0.0}
\definecolor{cv6}{rgb}{0.0,0.0,0.0}
\definecolor{cfv6}{rgb}{1.0,1.0,1.0}
\definecolor{clv6}{rgb}{0.0,0.0,0.0}
\definecolor{cv7}{rgb}{0.0,0.0,0.0}
\definecolor{cfv7}{rgb}{1.0,1.0,1.0}
\definecolor{clv7}{rgb}{0.0,0.0,0.0}
\definecolor{cv8}{rgb}{0.0,0.0,0.0}
\definecolor{cfv8}{rgb}{1.0,1.0,1.0}
\definecolor{clv8}{rgb}{0.0,0.0,0.0}
\definecolor{cv9}{rgb}{0.0,0.0,0.0}
\definecolor{cfv9}{rgb}{1.0,1.0,1.0}
\definecolor{clv9}{rgb}{0.0,0.0,0.0}
\definecolor{cv0v1}{rgb}{0.0,0.0,0.0}
\definecolor{cv2v3}{rgb}{0.0,0.0,0.0}
\definecolor{cv4v5}{rgb}{0.0,0.0,0.0}
\definecolor{cv6v7}{rgb}{0.0,0.0,0.0}
\definecolor{cv8v9}{rgb}{0.0,0.0,0.0}
\Vertex[style={minimum size=0.4cm,draw=cv0,fill=cfv0,text=clv0,shape=circle},LabelOut=false,L=\hbox{$2$},x=2.5cm,y=5.0cm]{v0}
\Vertex[style={minimum size=0.4cm,draw=cv1,fill=cfv1,text=clv1,shape=circle},LabelOut=false,L=\hbox{$1$},x=0.9549cm,y=4.5225cm]{v1}
\Vertex[style={minimum size=0.4cm,draw=cv2,fill=cfv2,text=clv2,shape=circle},LabelOut=false,L=\hbox{$10$},x=0.0cm,y=3.2725cm]{v2}
\Vertex[style={minimum size=0.4cm,draw=cv3,fill=cfv3,text=clv3,shape=circle},LabelOut=false,L=\hbox{$9$},x=0.0cm,y=1.7275cm]{v3}
\Vertex[style={minimum size=0.4cm,draw=cv4,fill=cfv4,text=clv4,shape=circle},LabelOut=false,L=\hbox{$8$},x=0.9549cm,y=0.4775cm]{v4}
\Vertex[style={minimum size=0.4cm,draw=cv5,fill=cfv5,text=clv5,shape=circle},LabelOut=false,L=\hbox{$7$},x=2.5cm,y=0.0cm]{v5}
\Vertex[style={minimum size=0.4cm,draw=cv6,fill=cfv6,text=clv6,shape=circle},LabelOut=false,L=\hbox{$6$},x=4.0451cm,y=0.4775cm]{v6}
\Vertex[style={minimum size=0.4cm,draw=cv7,fill=cfv7,text=clv7,shape=circle},LabelOut=false,L=\hbox{$5$},x=5.0cm,y=1.7275cm]{v7}
\Vertex[style={minimum size=0.4cm,draw=cv8,fill=cfv8,text=clv8,shape=circle},LabelOut=false,L=\hbox{$4$},x=5.0cm,y=3.2725cm]{v8}
\Vertex[style={minimum size=0.4cm,draw=cv9,fill=cfv9,text=clv9,shape=circle},LabelOut=false,L=\hbox{$3$},x=4.0451cm,y=4.5225cm]{v9}
\Edge[lw=0.04cm,style={color=cv0v1,},](v0)(v1)
\Edge[lw=0.04cm,style={color=cv2v3,},](v2)(v3)
\Edge[lw=0.04cm,style={color=cv4v5,},](v4)(v5)
\Edge[lw=0.04cm,style={color=cv6v7,},](v6)(v7)
\Edge[lw=0.04cm,style={color=cv8v9,},](v8)(v9)
\end{tikzpicture}}

\end{center}
\caption{A length 4 gyration orbit in $FPL_5$, with corresponding link patterns.}
\label{fig:fplex}
\end{figure}

As another example, the FPL in Figure~\ref{fig:fplex2} is in a gyration orbit of size $84$ ($=12\cdot 7$), while $({\rm FPL}_6$, {\rm Gyr}, $f)$ exhibits resonance with frequency $12$.  
So even though ${\rm Gyr}^{12}(A)\neq A$ for $A$ an FPL in this orbit, ${\rm rot}^{12}(f(A))=f(A)$.

\begin{figure}
\begin{center}
\includegraphics[scale=.25]{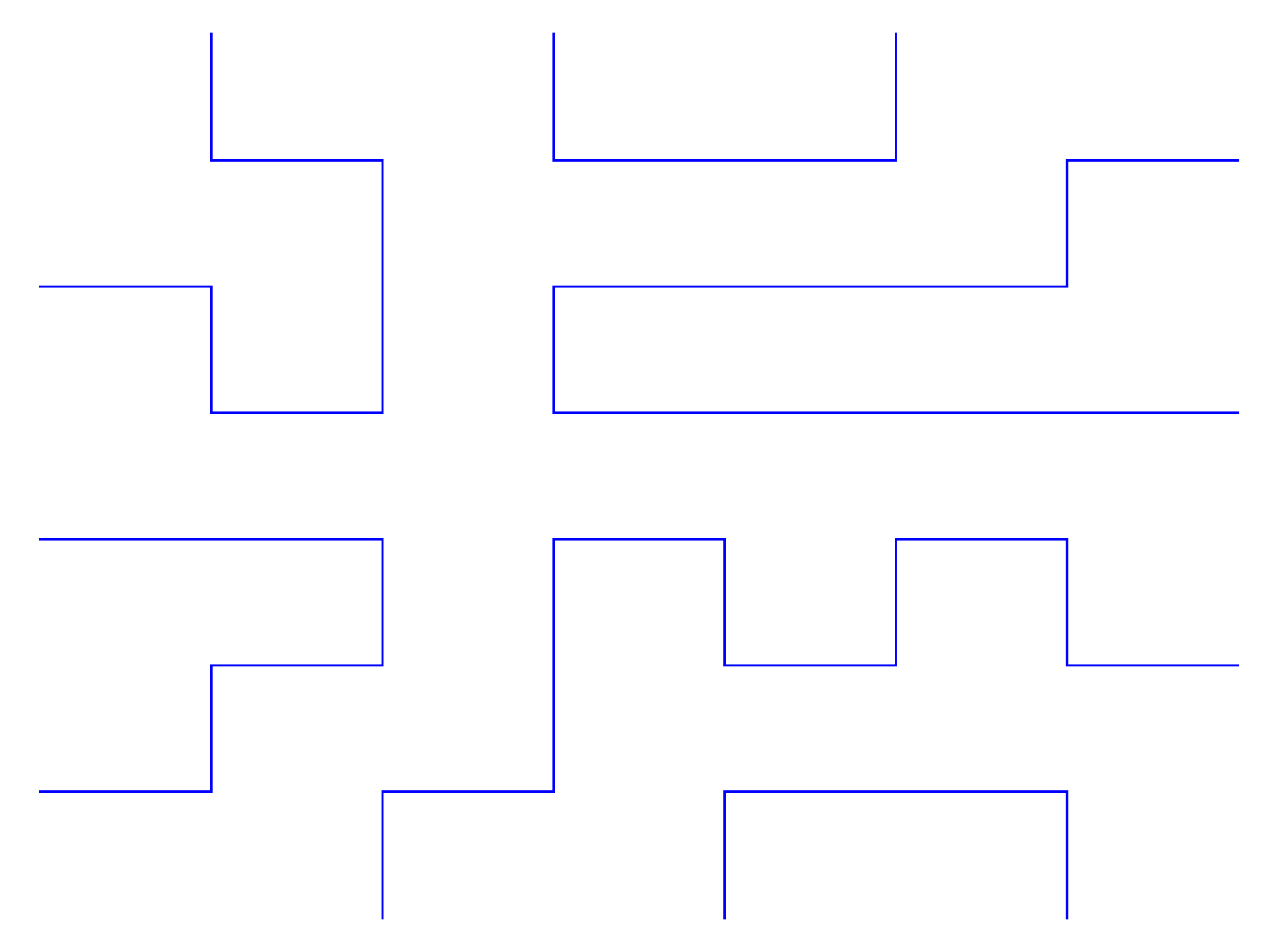}
\hspace{.2in}
\scalebox{.7}{
\begin{tikzpicture}
\useasboundingbox (0,0) rectangle (5.0cm,5.0cm);
\definecolor{cv0}{rgb}{0.0,0.0,0.0}
\definecolor{cfv0}{rgb}{1.0,1.0,1.0}
\definecolor{clv0}{rgb}{0.0,0.0,0.0}
\definecolor{cv1}{rgb}{0.0,0.0,0.0}
\definecolor{cfv1}{rgb}{1.0,1.0,1.0}
\definecolor{clv1}{rgb}{0.0,0.0,0.0}
\definecolor{cv2}{rgb}{0.0,0.0,0.0}
\definecolor{cfv2}{rgb}{1.0,1.0,1.0}
\definecolor{clv2}{rgb}{0.0,0.0,0.0}
\definecolor{cv3}{rgb}{0.0,0.0,0.0}
\definecolor{cfv3}{rgb}{1.0,1.0,1.0}
\definecolor{clv3}{rgb}{0.0,0.0,0.0}
\definecolor{cv4}{rgb}{0.0,0.0,0.0}
\definecolor{cfv4}{rgb}{1.0,1.0,1.0}
\definecolor{clv4}{rgb}{0.0,0.0,0.0}
\definecolor{cv5}{rgb}{0.0,0.0,0.0}
\definecolor{cfv5}{rgb}{1.0,1.0,1.0}
\definecolor{clv5}{rgb}{0.0,0.0,0.0}
\definecolor{cv6}{rgb}{0.0,0.0,0.0}
\definecolor{cfv6}{rgb}{1.0,1.0,1.0}
\definecolor{clv6}{rgb}{0.0,0.0,0.0}
\definecolor{cv7}{rgb}{0.0,0.0,0.0}
\definecolor{cfv7}{rgb}{1.0,1.0,1.0}
\definecolor{clv7}{rgb}{0.0,0.0,0.0}
\definecolor{cv8}{rgb}{0.0,0.0,0.0}
\definecolor{cfv8}{rgb}{1.0,1.0,1.0}
\definecolor{clv8}{rgb}{0.0,0.0,0.0}
\definecolor{cv9}{rgb}{0.0,0.0,0.0}
\definecolor{cfv9}{rgb}{1.0,1.0,1.0}
\definecolor{clv9}{rgb}{0.0,0.0,0.0}
\definecolor{cv10}{rgb}{0.0,0.0,0.0}
\definecolor{cfv10}{rgb}{1.0,1.0,1.0}
\definecolor{clv10}{rgb}{0.0,0.0,0.0}
\definecolor{cv11}{rgb}{0.0,0.0,0.0}
\definecolor{cfv11}{rgb}{1.0,1.0,1.0}
\definecolor{clv11}{rgb}{0.0,0.0,0.0}
\definecolor{cv0v11}{rgb}{0.0,0.0,0.0}
\definecolor{cv1v2}{rgb}{0.0,0.0,0.0}
\definecolor{cv3v4}{rgb}{0.0,0.0,0.0}
\definecolor{cv5v8}{rgb}{0.0,0.0,0.0}
\definecolor{cv6v7}{rgb}{0.0,0.0,0.0}
\definecolor{cv9v10}{rgb}{0.0,0.0,0.0}
\Vertex[style={minimum size=0.4cm,draw=cv0,fill=cfv0,text=clv0,shape=circle},LabelOut=false,L=\hbox{$2$},x=2.5cm,y=5.0cm]{v0}
\Vertex[style={minimum size=0.4cm,draw=cv1,fill=cfv1,text=clv1,shape=circle},LabelOut=false,L=\hbox{$1$},x=1.25cm,y=4.6651cm]{v1}
\Vertex[style={minimum size=0.4cm,draw=cv2,fill=cfv2,text=clv2,shape=circle},LabelOut=false,L=\hbox{$12$},x=0.3349cm,y=3.75cm]{v2}
\Vertex[style={minimum size=0.4cm,draw=cv3,fill=cfv3,text=clv3,shape=circle},LabelOut=false,L=\hbox{$11$},x=0.0cm,y=2.5cm]{v3}
\Vertex[style={minimum size=0.4cm,draw=cv4,fill=cfv4,text=clv4,shape=circle},LabelOut=false,L=\hbox{$10$},x=0.3349cm,y=1.25cm]{v4}
\Vertex[style={minimum size=0.4cm,draw=cv5,fill=cfv5,text=clv5,shape=circle},LabelOut=false,L=\hbox{$9$},x=1.25cm,y=0.3349cm]{v5}
\Vertex[style={minimum size=0.4cm,draw=cv6,fill=cfv6,text=clv6,shape=circle},LabelOut=false,L=\hbox{$8$},x=2.5cm,y=0.0cm]{v6}
\Vertex[style={minimum size=0.4cm,draw=cv7,fill=cfv7,text=clv7,shape=circle},LabelOut=false,L=\hbox{$7$},x=3.75cm,y=0.3349cm]{v7}
\Vertex[style={minimum size=0.4cm,draw=cv8,fill=cfv8,text=clv8,shape=circle},LabelOut=false,L=\hbox{$6$},x=4.6651cm,y=1.25cm]{v8}
\Vertex[style={minimum size=0.4cm,draw=cv9,fill=cfv9,text=clv9,shape=circle},LabelOut=false,L=\hbox{$5$},x=5.0cm,y=2.5cm]{v9}
\Vertex[style={minimum size=0.4cm,draw=cv10,fill=cfv10,text=clv10,shape=circle},LabelOut=false,L=\hbox{$4$},x=4.6651cm,y=3.75cm]{v10}
\Vertex[style={minimum size=0.4cm,draw=cv11,fill=cfv11,text=clv11,shape=circle},LabelOut=false,L=\hbox{$3$},x=3.75cm,y=4.6651cm]{v11}
\Edge[lw=0.04cm,style={color=cv0v11,},](v0)(v11)
\Edge[lw=0.04cm,style={color=cv1v2,},](v1)(v2)
\Edge[lw=0.04cm,style={color=cv3v4,},](v3)(v4)
\Edge[lw=0.04cm,style={color=cv5v8,},](v5)(v8)
\Edge[lw=0.04cm,style={color=cv6v7,},](v6)(v7)
\Edge[lw=0.04cm,style={color=cv9v10,},](v9)(v10)
\end{tikzpicture}}

\end{center}
\caption{A $6\times 6$ FPL with gyration orbit of length 84, and its link pattern}
\label{fig:fplex2}
\end{figure}

\smallskip
Finally, in Problems~\ref{prob:spro} and~\ref{prob:tsscpp} below, we reformulate some observations from~\cite{prorow} in terms of resonance; for additional details, see~\cite[Sections~8.3 and~8.4]{prorow}.

Fully-packed loops are known to be in bijection with \emph{alternating sign matrices}~\cite{wieland,ProppManyFaces}. Alternating sign matrices were introduced by D.\ Robbins--H.~Rumsey \cite{RobbinsRumsey} as part of their study of the lambda-determinant. With W.~Mills \cite{MRRASM}, they then conjectured an enumeration for $n\times n$ alternating sign matrices, which was proved by D.~Zeilberger~\cite{zeilberger} and G.~Kuperberg~\cite{kuperberg} (cf.~\cite{BressoudBook} for a detailed exposition of this history).

%\begin{remark}
%\label{remark:asmposetgyr}
%The name of the toggle group element $\mathrm{Gyr}$ of Definition~\ref{def:gyr} comes from a connection to the gyration action discussed above on {fully-packed loop configurations}. 
%The connection is that there is a poset $\mathbf{A}_n$ whose order ideals are in bijection with fully-packed loop configurations, and for which the gyration action of Definition~\ref{def:fplgyr} is the action of $\mathrm{Gyr}$. For details, see \cite[Section~8]{prorow} and \cite{razstrogrow}.
%\end{remark}

%Recall from Remark~\ref{remark:asmposetgyr} that 
There is a poset $\mathbf{A}_n$ whose order ideals are in bijection with $n\times n$ alternating sign matrices (denote this set as $\mathrm{ASM}_n$), such that gyration of Definition~\ref{def:fplgyr} is equivalent to the action of the toggle group element $\mathrm{Gyr}$ of Definition~\ref{def:gyr}. For details, see \cite[Section~8]{prorow} and \cite{razstrogrow}.
Another element, $\mathrm{SPro}$, of the toggle group on $\mathbf{A}_n$ was introduced in \cite[Definition~8.14]{prorow}. It is shown in \cite[Theorem~8.15]{prorow} that the orbit of the empty order ideal in $J(\mathbf{A}_n)$ under $\mathrm{SPro}$ has cardinality $3n-2$. Further data contained in \cite[Figure~22]{prorow} leads us to propose the following.

\begin{problem}
\label{prob:spro}
Construct a natural map $f$ such that
$({\rm ASM}_n, \mathrm{SPro}, f)$ exhibits resonance with frequency $3n-2$.
\end{problem}

Similarly, there is a poset $\mathbf{T}_n$ whose order ideals are in bijection with \emph{totally symmetric self-complementary plane partitions} inside a $2n\times 2n\times 2n$ box (denote this set as $\mathrm{TSSCPP}_n$). For details, see \cite[Section~8]{prorow} and \cite{tetra,razstrogrow}. It is shown in \cite[Theorem~8.19]{prorow} that the cardinality of the rowmotion-orbit of the empty order ideal in $J(\mathbf{T}_n)$ is $3n-2$. Further data contained in \cite[Figure~22]{prorow} suggests the following.
\begin{problem}
\label{prob:tsscpp}
Construct a natural map $f$ such that
$({\rm TSSCPP}_n, \row, f)$ exhibits resonance with frequency $3n-2$.
\end{problem}

We suspect that a solution to the above problems would be a major step towards exhibiting an explicit bijection between ${\rm ASM}_n$ and $\mathrm{TSSCPP}_n$, which are known (non-bijectively) to be equinumerous~\cite{ANDREWS_TSSCPP,zeilberger,kuperberg}.

\section*{Acknowledgments} This work began at the ``Dynamical Algebraic Combinatorics'' workshop at the American Institute of Mathematics (AIM) in March 2015. The authors are grateful for many inspiring conversations with the participants. The authors would like to thank Jim Propp, Tom Roby, and Nathan Williams for their organizational efforts, as well as AIM for funding this workshop and the AIM staff for their hospitality.
%Additional thanks to J.~Propp for coining the term ``resonance.''

We thank two anonymous referees for helpful suggestions that significantly improved the exposition of this paper.

KD was supported by RTG grant NSF/DMS-1148634. OP was supported by an Illinois Distinguished Fellowship and an NSF Graduate Research Fellowship. JS was supported by a National Security Agency Grant (H98230-15-1-0041), the North Dakota EPSCoR National Science Foundation Grant (IIA-1355466), and the NDSU Advance FORWARD program sponsored by National Science Foundation grant (HRD-0811239).

\end{document}